\documentclass[twoside,11pt,reqno]{amsart}

\DeclareSymbolFont{symbols}{OMS}{cmsy}{m}{n}
\usepackage{fullpage}

\usepackage[utf8]{inputenc}
\usepackage[T1]{fontenc}

\usepackage{mathrsfs}
\usepackage{mathtools}

\mathtoolsset{showonlyrefs=true}

\usepackage{amsmath,amssymb,amsthm}
\usepackage{amsfonts}
\usepackage{relsize}

\usepackage{enumerate}
\usepackage{bbm}
\usepackage{graphicx,tikz}
\usetikzlibrary{arrows,scopes}
\usetikzlibrary{calc,cd}
\usepackage{setspace}

\usepackage{hyperref}
\usepackage[alphabetic,initials]{amsrefs}

\newtheorem{thm}{Theorem}[section]

\newtheorem{lem}[thm]{Lemma}
\newtheorem{cor}[thm]{Corollary}
\newtheorem{prop}[thm]{Proposition}
\theoremstyle{definition}
\newtheorem{defi}[thm]{Definition}

\newtheorem{question}[thm]{Question}

\newtheorem{example}[thm]{Example}
\theoremstyle{remark}
\newtheorem{rmk}[thm]{Remark}

\newcommand{\Hil}{\mathcal{H}}
\newcommand{\slot}{{~\cdot~}}

\newcommand{\Spin}{\mathop{\mathsf{Spin}}}
\newcommand{\SO}{\mathop{\mathsf{SO}}}
\newcommand{\SL}{\mathop{\mathsf{SL}}}

\newcommand{\SU}{\mathop{\mathsf{SU}}}

\newcommand{\Mob}{\mathsf{M\ddot ob}}

\newcommand{\Sc}[1][]{\mathbb{S}^{1#1}}

\newcommand{\C}{\mathcal{C}}

\newcommand{\cF}{\mathcal{F}}

\newcommand{\cB}{\mathcal{B}}

\newcommand{\cD}{\mathcal{D}}

\newcommand{\cC}{\mathcal{C}}
\newcommand{\cS}{\mathcal{S}}

\newcommand{\A}{\mathcal{A}}
\newcommand{\cI}{\mathcal{I}}

\newcommand{\N}{\mathcal{N}}

\newcommand{\RR}{\mathbb{R}}
\newcommand{\TT}{\mathbb{T}}

\newcommand{\CC}{\mathbb{C}}

\newcommand{\ZZ}{\mathbb{Z}}
\newcommand{\NN}{\mathbb{N}}

\DeclareMathOperator{\Diff}{Diff}
\DeclareMathOperator{\End}{End}
\DeclareMathOperator{\Rep}{Rep}

\DeclareMathOperator{\Hom}{Hom}

\DeclareMathOperator{\Vir}{Vir}

\DeclareMathOperator{\diag}{diag}

\DeclareMathOperator{\id}{id}

\DeclareMathOperator{\tr}{tr}
\DeclareMathOperator{\B}{B}

\DeclareMathOperator{\sign}{sign}

\DeclareMathOperator{\Aut}{Aut}

\newcommand{\e}{\mathrm{e}}

\newcommand{\ima}{\mathrm{i}}

\newcommand{\op}{\mathrm{op}}

\usepackage{xspace}
\DeclareRobustCommand{\eg}{e.g.\@\xspace}
\DeclareRobustCommand{\cf}{cf.\@\xspace}
\DeclareRobustCommand{\ie}{i.e.\@\xspace}
\makeatletter
\DeclareRobustCommand{\etc}{%
    \@ifnextchar{.}%
        {etc}%
        {etc.\@\xspace}%
}

\DeclareMathOperator{\Conv}{Conv}
\DeclareMathOperator{\QuOp}{QuOp}
\DeclareMathOperator{\Exp}{Exp}
\DeclareMathOperator{\aarg}{arg}
\DeclareMathOperator{\Pic}{Pic}
\DeclareMathOperator{\Ni}{Ni}
\DeclareMathOperator*{\spec}{spec}
\DeclareMathOperator*{\Dim}{Dim}
\DeclareMathOperator{\Dih}{Dih}
\DeclareMathOperator{\Out}{Out}
\DeclareMathOperator*{\TY}{\mathcal{TY}}
\DeclareMathOperator*{\MP}{\mathcal{MP}}

\DeclareMathOperator{\Vect}{Hilb}

\DeclareMathOperator*{\Hilb}{Hilb}
\DeclareMathOperator{\Irr}{Irr}
\newcommand{\tRep}[1]{#1\text{--}\Rep}
\newcommand{\CS}{/\!/}
\newcommand{\tu}{\mathbbm{1}} %

\newcommand{\umtc}[1]{\cC_{#1}}
\newcommand{\bi}[2]{\langle #1,#2\rangle}
\newcommand{\bii}[2]{b( #1,#2)}
\newcommand{\cG}{\mathcal{G}}
\newcommand{\bim}[4][]{{}\prescript{\vphantom{#1}}{#2}{#3}^{#1}_{#4}}
\newcommand{\s}[2]{{}_{#1}\mathcal{C}_{#2}}
\newcommand{\NNs}{\s{\N}{\N}}

\renewcommand{\N}{N}

\newcommand{\rev}[1]{{#1}^\mathrm{rev}}

\newcommand{\hb}{\varepsilon} %
\renewcommand{\slot}{\,\cdot\,}
\newcommand{\bislotslot}{\bi\cdot\cdot}

\newcommand{\lattimes}{}%
\newcommand{\myemph}[1]{\emph{#1}}

\newcommand{\tikzmath}[2][0.50]
{\vcenter{\hbox{\begin{tikzpicture}[scale=#1] #2\end{tikzpicture}}}
}

\begin{document}
\date{\today}
\dateposted{\today}
\title[Realizability of TY Categories]{Conformal Net Realizability of 
  Tambara--Yamagami Categories
 and Generalized Metaplectic Modular Categories
}
\address{Department of Mathematics, Morton Hall 321, 1 Ohio University, Athens, OH 45701, USA}
\author{Marcel Bischoff}
\email{bischoff@ohio.edu}
\email{marcel@localconformal.net}
\thanks{
  Supported in part by NSF Grant DMS-1362138 and NSF Grant DMS-1700192/1821162}
\dedicatory{This paper is dedicated to Karl-Henning Rehren on the occasion of his 60th birthday}
\begin{abstract}
We show that all isomorphism classes of even rank Tambara--Yamagami categories
arise as $\ZZ_2$-twisted representations of conformal nets.
As a consequence, we show that their Drinfel'd centers are realized 
by (generalized) orbifolds of conformal nets associated with (self-dual) lattices.
The quantum double subfactors of even rank Tambara--Yamagami 
categories are Bisch--Haagerup subfactors and
we describe their (dual) principal graphs.

For every abelian group of odd order the Drinfel'd centers of the associated Tambara--Yamagami categories give a fusion ring generalizing 
the Verlinde ring $\Spin(2n+1)_2$ in the case of $\ZZ_{2n+1}$.
We classify all generalized metaplectic modular categories, i.e.\ unitary modular tensor category with those fusion rules and show that they are realized as $\ZZ_2$-orbifolds of conformal nets associated with lattices.  
We further show that twisted doubles of generalized dihedral groups of abelian 
groups of odd order are group-theoretical generalized metaplectic modular categories and vice versa.

We give some examples of twisted orbifolds of conformal nets and 
show how generalized metaplectic modular categories arise
by condensation of simpler ones.

\end{abstract}

\maketitle
\tableofcontents

\section{Introduction}

Local conformal nets on the circle axiomatizing chiral conformal field theory
using von Neumann algebras. 
The local algebras turn out to be factors, \ie von Neumann algebras with 
trivial centers and conformal nets give rise to subfactors, 
\ie unital inclusion of factors in many different ways.
Via their representation theory, conformal nets also give rise to (braided) 
C$^{\ast}$-tensor categories.

In particular, in \cite{KaLoMg2001}, Kawahigashi, Longo, and Müger have 
introduced the notion of a 
\myemph{completely rational conformal net} and showed that the 
representation category of such a net is a unitary modular tensor category. 
A natural question is if all unitary modular tensor categories arise this way
in the following sense: 
Given any abstract unitary modular tensor category $\cC$,
is there always a completely rational conformal net $\A$ which realizes $\cC$, 
\ie $\Rep(\A)$ is braided equivalent to $\cC$?
Finding a solution for a family of categories goes under the
name of \emph{reconstruction program}.
One difficulty is that such a net $\A$ realizing $\cC$ (if it exists) is far 
from unique.
For example, every net $\A\otimes \cB$ with $\cB$ a \myemph{holomorphic net},
\ie a completely rational conformal net $\cB$ with trivial representation 
category $\Rep(\cB)\cong \Hilb$, also realizes $\cC$.
Every even self-dual positive lattice $\Gamma$ gives a holomorphic net 
$\A_\Gamma$
but the classification of even self-dual positive lattices itself is a hopeless
problem.

If $\cF$ is a unitary fusion category, 
which means that it does not need to be braided, 
then its Drinfel'd center $Z(\cF)$ is a 
unitary modular tensor category \cite{Mg2003II}.
In this case, we may ask as---a special case of the reconstruction program---if 
all Drinfel'd centers of fusion categories are realized by conformal nets,
\ie if for any unitary fusion category $\cF$ there is a 
completely rational conformal net $\A$ which realizes $Z(\cF)$.
A positive answer to this question would imply that all 
unitary fusion categories and all finite index finite depth subfactors arise 
from conformal nets using \emph{higher representation theory}
\cite{Bi2015,Bi2015VFR}.

If $\A$ is a net realizing $Z(\cF)$ for some unitary fusion category $\cF$, 
we call $\A$ a \myemph{quantum double net}.
In this case, there is a holomorphic conformal net $\cB$, 
and a proper action of the fusion ring hypergroup $K_\cF$ on $\cB$,
such that $\A$ is the fixed point net $\cB^{K_\cF}$ \cite{Bi2016}. 
The unitary fusion category $\cF$ can be reconstructed as the category of 
solitons coming from $\alpha$-induction applied to the inclusion 
$\A\subset \cB$. 
Conversely, if a finite hypergroup $Q$ acts properly on a holomorphic net $\cB$,
then the fixed point net $\A^Q$ is a quantum double net with $\Rep(\A^Q)$ 
braided equivalent to $Z(\cF)$ for a unitary fusion category $\cF$ which is a 
\myemph{categorification} of $Q$, \ie $K_\cF\cong Q$.

The main goal of this paper is to establish quantum double nets and 
therefore a reconstruction for a certain family of unitary fusion categories, 
namely so-called Tambara--Yamagami categories of even rank.

Namely, we prove that for every unitary fusion category $\cF$ with
$\Irr(\cF)=G\cup \{\rho\}$
for some finite group $G$ of odd order and fusion rules
\footnote{these fusion rules are also denoted by $G+0$ in \cite{EvGa2014}}
\begin{align}  
  [\rho] [\rho]&=\sum_{g\in G} [g]\,,&
  [g][\rho]&=[\rho] [g] =[\rho]\,,&
  [g] [h]& =[gh]\,,
\end{align}
there is a net $\A$ with $\Rep(\A)$ braided equivalent to $Z(\cF)$.
We note that Tambara--Yamagami categories are classified \cite{TaYa1998}
and that $G$ is necessarily abelian.

It is conjectured that unitary fusion categories are coming from models in 
low-dimensional physics.
In particular, it is believed, that one can use the data to obtain a model in 
statistical physics whose critical limit is a conformal field theory.

One might ask:
\begin{question}
  \label{quest:TY}
  Where do Tambara--Yamagami categories come from?
\end{question}
One of the easiest rational CFT models are in physical language sigma models 
 or chiral Wess--Zumino--Witten models
with target space a (metric) torus,
\ie euclidean space $\RR^n$ compactified by a necessarily even positive lattice $L$.
Here $L$ can be seen as the level $H_+^4(B\TT^n,\ZZ)$ for the $n$-torus $\TT^n$,
see \cite{He2016pp}.
These models can be described by a conformal net $\A_L$.

Indeed, the reconstruction of Tambara--Yamagami categories can be done by taking $\ZZ_2$-orbifolds of the nets $\A_L$ and we can give an answer to 
Question \ref{quest:TY} for the case of even rank: 
\begin{thm}[Theorem \ref{thm:AllOddTYs}]
  Tambara--Yamagami categories of even rank arise as the category
  of $\ZZ_2$-twisted representations of a net $\A_L$.
\end{thm}
Using our reconstruction result, we can make several interesting observations and 
connections.
We introduce the notion of a \emph{generalized metaplectic modular category} 
associated with an abelian group $A$ of odd order which is a 
unitary modular tensor categories with fusion rules depending on $A$ 
which generalizes the Verlinde fusion rules of $\Spin(2k+1)_2$ for the case 
$A=\ZZ_{2k+1}$, see Definition \ref{defi:GeneralizedMetaplecticModularCategory}.
We generalize the classification of metaplectic modular categories 
up to braided equivalence in \cite{ArChRoWa2016} to 
generalized metaplectic modular categories, see 
  Theorem \ref{thm:MPclass1}, \ref{thm:MPclass2}, 
and show that generalized metaplectic modular categories 
are determined by its modular data, see Theorem \ref{thm:MPST}.
Again, we have a similar reconstruction result:
\begin{thm}[Theorem \ref{thm:MPRealization}]
    All odd generalized metaplectic modular categories are be realized by 
    $\ZZ_2$-orbifolds of conformal nets associated with even lattices.
\end{thm}
We can use the classification to show that $\Rep(\A_{\Spin(2p+1)_2})$ of 
the loop group net $\A_{\Spin(2p+1)_2}$ is braided equivalent
to the unitary modular tensor category $\C_{\Spin(2p+1)_2}$ for all $p\in\NN$
and is \myemph{modular} in the sense of \cite{KaLo2005}.

We give some results on how (generalized) metaplectic modular categories can be obtained from simpler building blocks via condensation, 
\ie as the category of local modules with respect to a commutative Q-system. 
We show how metaplectic modular categories can often be realized as condensation
of $\Spin(p)_2$, their reverses, and semion categories. 
This is always possible if no prime factor of $n$ is a Pythagorean prime.

In Section \ref{sec:Dihedral}, we give several relations of 
generalized metaplectic modular categories and Tambara--Yamagami categories to 
  \emph{generalized dihedral groups}, 
\ie groups of the form $\Dih(A):=A\rtimes \ZZ_2$ for some abelian group $A$, 
where $\ZZ_2$ acts by $a\mapsto a^{-1}$. 
   
We introduce a generalization of Tambara--Yamagami categories in 
Subsection \ref{ssec:GenTY}. 
They are extensions of pointed unitary fusion categories of 
    generalized dihedral groups and are Morita equivalent to 
    generalized metaplectic modular categories.

In Subsection \ref{ssec:DoublesOfGenDihGroups}, we show that any 
twisted quantum double of a generalized dihedral group $\Dih(A)$ with $|A|$ odd
is a generalized metaplectic modular category.
Conversely, a generalized metaplectic modular category
which is the Drinfel'd center of a unitary fusion category
is braided equivalent to the twisted double
of a generalized dihedral group.
In particular, a generalized metaplectic modular category
is group theoretical if and only if it is Witt trivial, 
the Drinfel'd center of a unitary fusion category.

The Longo--Rehren subfactor associated with (odd) Tambara--Yamagami categories
are Bisch--Haagerup subfactors $M^H\subset M\rtimes K$. 
We determine the associated $G$-kernel for 
  $G=\langle H,K\rangle$ and describe the principal and dual principal graphs,
see Subsection \ref{ssec:QuantumDoubleAndBischHaagerup}.

We show that Drinfel'd center of Tambara--Yamagami categories based on an 
abelian group $A$ of odd order can be realized as $\ZZ_2$-twisted orbifolds of  
$\Dih(A)$-fixed point nets of holomorphic nets associated with 
even self-dual lattices, see Subsection \ref{ssec:TYareTwistedOrbifolds}.

As an outlook, and a posterior motivation of this work, 
we remark that the just mentioned result implies that
the two modular data from the unitary modular tensor categories
$\umtc{\Spin(2n+1)_2}$ and the (twisted) quantum double of $D_m$ for $m$ odd, 
respectively, which serve as an ingredient for a \emph{grafting} in 
\cite{EvGa2011} are both the modular data of 
generalized metaplectic modular categories. 
In particular, the modular data \cite[Proposition 7b]{EvGa2014} 
which conjectural is the modular data for $G+n$ near group categories 
\cite{EvGa2014} with 
$n=|G|$ factorizes as $(S,T)=(S'\otimes S_{(G,q)}, T'\otimes T_{(G,q)})$,
where  $(S_{(G,q)}, T_{(G,q)})$ is the Weil representation associated with
$G$ and a non-degenerate quadratic form $q$, or equivalently the 
modular data of a pointed UMTC $\cC(G,q)$. 
The (interesting) factor $(S',T')$ can be thought to be build of the 
modular data of two generalized metaplectic modular categories 
associated with groups $G$ and $G'$, respectively, where $|G'|=|G|+4$. 
For example, for $G=\ZZ_3\times\ZZ_3$ and $G'=\ZZ_{13}$ one gets 
the modular data $(S',T')$ of a factor of the quantum double of a 
$\ZZ_3\times\ZZ_3+9$ near group category.
The generalized metaplectic modular category associated with $G'$ is a 
twisted double of $S_3\cong \Dih(\ZZ_3\times \ZZ_3)$. 
This modular data also corresponds to the double of the even part of the 
Haagerup subfactor, see also \cite[Remark 5.14]{Bi2016}.
The argument can be generalized to other near group categories, and 
near group categories with a certain Lagrangian correspond to
Izumi--Haagerup categories.

In Section \ref{sec:GeneralizedOrbifoldsAndDefects} we give a relation of the 
constructed models to generalized orbifolds and topological defects.
We show that some of the examples allow to construct twisted orbifolds 
by generalized dihedral groups, see Subsection \ref{ssec:HolomorphicNets}.
We briefly discuss the harmonic analysis for actions of the 
Tambara--Yamagami fusion hypergroup on the obtained models.
Finally, we discuss generalized Kramers--Wannier dualities as in
\cite{FrFuRuSc2004} using the work \cite{BiKaLoRe2014}.
Namely, let $\cF$ be a fusion category, then following \cite{FrFuRuSc2004} a 
simple object $X$ gives rise to a \myemph{duality defect} 
if every simple subobject of 
$X\otimes \bar X$.
Therefore the generating object of a Tambara--Yamagami category describes a 
generalized Kramers--Wannier duality.
We explain this in the setting of local conformal nets
in Subsection \ref{ssec:KramerWannier}.

\subsection*{Acknowledgements}
The author is grateful to Richard Ng for discussions, in particular, 
for informing him about Remark \ref{rmk:H3ab},
to Theo Johnson-Freyd for discussions about orbifolds, 
Zhengwei Liu for discussions about Bisch--Haagerup subfactors,
and Masaki Izumi for discussions about his work.
The author also would like to thank Luca Giorgetti and Henry Tucker for useful
comments on earlier versions of this manuscript.
Part of this work was initiated while the author visited the
Hausdorff Trimester Program ``Von Neumann Algebras'' and the author is grateful
for the hospitality of the Hausdorff Research Institute for Mathematics (HIM)
in Bonn.

\section{Preliminaries}
\label{sec:Preliminaries}

\subsection{Unitary fusion categories and unitary modular tensor categories}
\label{ssec:UFCsAndUMTCs}
A \myemph{unitary fusion category} $\cF$ is a rigid C${}^{\ast}$-tensor category
with simple tensor unit $\tu$, \ie $\Hom(\tu,\tu)\cong \CC$, 
such that set of isomorphism classes of irreducible or simple objects 
$\Irr(\cF)$ is finite.
We may assume $\cF$ to be strict, indeed in most of the cases $\cF$
is a full and replete subcategory of $\End_0(N)$, the category of endomorphisms
of a factor type III $N$ which have finite statistical dimension.

We say and object $\bar\rho\in\cF$ is a dual/conjugate for $\rho\in\cF$,
if there are $R\in\Hom(\tu,\bar\rho\otimes \rho)$ and 
$\bar R\in\Hom(\tu,\rho\otimes\bar\rho)$ such that
the \myemph{conugate equation} holds:
\begin{align}
  (\bar R^\ast \otimes 1_\rho)(1_\rho\otimes R)
    &=1_\rho\,; &
  ( R^\ast \otimes 1_{\bar\rho})(1_{\bar\rho}\otimes \bar R)
    &=1_{\bar\rho}\,.
\end{align}
The pair $(R,\bar R)$ is called \myemph{normalized} if $\|R\|=\|\bar R\|$.
Then $d_{(R,\bar R)}\cdot 1_{\tu}=R^\ast R=\bar R^\ast \bar R$ 
defines a positive number $d_{(R,\bar R)}$
and $d\rho:=\inf_{(R,\bar R)} d_{(R,\bar R)}$, where the 
infimum runs over all normalized solution of the conjugate equation is called the 
\myemph{(statistical) dimension} of $\rho$.
A solution $(R,\bar R)$ of the conjugate eqution is called \myemph{standard}
if $d\rho\cdot 1_\tu=R^\ast R$.
A choice of standard solutions equips the unitary fusion category $\cF$ 
with an (essentially) unique spherical structure and the 
dimension $d\rho \in [1,\infty)$ equals the
Frobenius--Perron dimension of $\rho$
\cite{LoRo1997}.

Let $\rho\in \cF$ be an irreducible object of a unitary fusion category $\cF$. 
The \myemph{Frobenius--Schur indicator} $\nu_\rho$ is defined to be 
$\nu_\rho=0$ if $[\rho]\neq[\bar\rho]$.
If $[\rho]=[\bar\rho]$, we may choose $\bar\rho=\rho$ and there is a unique 
$\nu_\rho=\pm1$, such that there is an $R\in\Hom(\id,\rho\rho)$ and $(R,\bar R)$
with $\bar R =\nu_\rho R$ is a solution to the conjugate equation 
\cite{LoRo1997}.

A \myemph{braided unitary fusion category} $\cC$ is a unitary fusion category
with a unitary braiding, \ie a natural family of unitaries 
$\{\varepsilon(\rho,\sigma)\in\Hom(\rho\otimes\sigma,\sigma\otimes\rho)
  \}_{\rho,\sigma\in \cC}$, fulfilling the usual hexagon identities.
Because of its spherical structure a braided unitary fusion category is a 
premodular category, \ie a ribbon fusion category \cite{Mg2003-MC}.
For $\cC$ a braided fusion category, we denote by $\rev{\cC}$ the 
braided fusion category with the opposite braiding 
$\varepsilon^-(\rho,\sigma)= \varepsilon(\sigma,\rho)^\ast$.

Let $\cC$ be a braided unitary fusion category. 
For a full subcategory $\cD$ we define the \myemph{M\"uger centralizer}
\begin{align}
  \cD'\cap \cC=\{\rho\in \cC: \varepsilon(\rho,\sigma)
    =\varepsilon(\sigma,\rho)^\ast \text{ for all }\sigma\in \cD\}\,.
\end{align}
The category $\cC$ is called \myemph{non-degenerately braided} if the 
Rehren--M\"uger center $\cC'\cap \cC$ is equivalent to the trivial
(braided) fusion category, %
\ie $\cC'\cap \cC\cong \Vect$.
A non-degenerate braided unitary fusion category $\cC$ is a 
\myemph{unitary modular tensor category (UMTC)} \cite{Re1989,Mg2003-MC}.

From every unitary fusion category $\cF$, we get the unitary modular tensor 
category $Z(\cF)$, the \myemph{(unitary) Drinfel'd center} of $\cF$.
A braided unitary fusion category is a unitary modular tensor category
if and only if $Z(\cC)$ is braided equivalent to $\cC\boxtimes\rev{\cC}$ 
\cite[Corollary 7.11]{Mg2003II}.

We say two unitary modular tensor categories $\cC$ and $\cD$ are 
\myemph{(unitarily) Witt equivalent} if there are unitary fusion categories 
$\cF$ and $\cG$, such that $\cC \boxtimes Z(\cF)$ is braided equivalent 
to $\cD\boxtimes Z(\cG)$. 
Witt equivalence is an equivalence relation and the equivalence classes 
form an abelian group under the multiplication $[\cC][\cD]=[\cC\boxtimes \cD]$
with identity $[\Vect]=[Z(\cF)]$ and inverse $[\cC]^{-1}=[\rev{\cC}]$, see  
\cite{DaMgNiOs2013}.

Let $\cC$ be a unitary braided fusion category of rank $n=|\Irr(\cC)|$.
We define for $\lambda,\mu\in \Irr(\cC)$
\begin{align}
	Y_{\lambda\mu}&=
  \tr_{\lambda\otimes\mu}\left[ 
  \varepsilon(\lambda,\mu)^\ast
  \varepsilon(\mu,\lambda)^\ast
  \right]
  =
  \tikzmath[.3]{
	\begin{scope}[yscale=-1]
		\draw[thick] (-2.5,0) arc (180:60:1.5);
		\draw[thick] (-2.5,0) node [left] {$\bar\lambda$} arc (180:390:1.5);
		\draw[thick]  (2.5,0) node [right] {$\bar\mu$} arc (0:210:1.5);
		\draw[thick] (2.5,0) arc (360:240:1.5);
	\end{scope}
	}
  \,;
  \\
  \omega_\lambda&=
  (d\lambda)^{-1} \cdot \tr_{\lambda\otimes\lambda}
  \left[
  \varepsilon(\lambda,\lambda)
  \right]
  \,,~
  \omega_\lambda\cdot \id_\lambda 
  =
  \tikzmath[.3]{
	\begin{scope}[yscale=-1]
		\draw[thick]  (0,-1.5)node [above] {$\lambda$}--(0,-0.7)
  		(0,0)--(0,1.5) node [below] {$\lambda$};
		\draw [thick] (0,0) arc  (180:490:.7);
	\end{scope}
	}
\intertext{and the following $n \times n$-matrices}
  \label{eq:ST}
  S_{\lambda\mu}&=(\dim \NNs)^{-\frac 12} Y_{\lambda,\mu}\,,
  &T_{\lambda\mu}&= \e^{-\pi \ima c_\mathrm{top}/12} \delta_{\lambda\mu}\omega_\lambda\,,
\end{align}
where the \myemph{topological central charge} 
$c_\mathrm{top}\equiv c_\mathrm{top}\pmod 8$ is defined  by 
\begin{align}
  c_\mathrm{top}&=\frac{4\aarg(z)}{\pi} \quad \text{where}\quad %
  z=\sum_{\rho\in\Irr(\cC)} (d\rho)^2\cdot \omega_\rho\,. 
\end{align}
The matrices $S$ and $T$ obey the relations of the 
\myemph{partial Verlinde modular algebra}:
  $TSTST=S$,
  $CTC=T$, and $CSC=S$
\cite{Re1989,BcEvKa1999}, 
where $C_{\mu\nu}=\delta_{\mu,\bar\nu}$ is the \myemph{charge conjugation matrix} .
The matrix $S$ is unitary if and only if $\cC'\cap \cC$ is trivial 
\cite{Re1989}, \ie $\cC$ is a unitary modular tensor category. 
In this case, we have $(ST)^3=C=S^2$, thus $(S,T)$  
gives a unitary representation of $\SL(2,\ZZ)$ on $\cC^{|\Irr(\cC)|}$.

The pair $(S,T)$ is called the \myemph{modular data} associated with
a unitary modular tensor category $\cC$
and we have the \myemph{Verlinde formula} \cite{Re1989}:
\begin{align} 
  \label{eq:Verlinde}
  N_{i,j}^k=\sum_\ell \frac{S_{j,\ell}S_{i,\ell}\bar S_{k,\ell}}{S_{0,\ell}}\,.
\end{align}
We say two modular data $(S,T)$ and $(S',T')$ are equivalent, if there is a
a third root of unity $\zeta$ and a 
bijection $\sigma$ from the index set of $(S,T)$ to the index set of $(S',T')$ 
which fixes first element (corresponding to $\tu$) of the index set, 
such that $S'_{\sigma (i),\sigma (j)} = S_{i,j}$ and 
$T'_{\sigma(i),\sigma(j)}=\zeta T_{i,j}$ for all possible indices. 
The modular data assocatiated with $\cC$ up to equivalence is an invariant of 
the modular tensor category $\cC$.

Let $\cC$ be a unitary modular tensor category and $\rho\in \Irr(\cC)$.
The Frobenius--Schur indicator $\nu_\rho$ is in terms of $(S,T)$ by
Bantay's formula \cite{Ba1997,GiRe2016}:
\begin{align}
  \label{eq:Bantay}
  \nu_\rho= \sum_{\sigma,\tau} S_{\sigma,\tu}S_{\tau,\tu} N_{\sigma,\tau}^\rho \frac{T_{\sigma,\sigma}^2}{T_{\tau,\tau}^2}\in\{0,\pm 1\}\,.
\end{align}

\subsection{Pointed unitary fusion categories}
\label{ssec:Pointed}
A unitary fusion category $\cF$ is called \myemph{pointed} if every 
irreducible/simple object $\rho\in\cC$ is invertible, 
\ie has dimension $d\rho=1$.
Then $G=\Irr(\cC)$ forms a finite group under the multiplication 
$[\rho][\sigma] = [\rho\otimes\sigma]$
and $\cF$ is tensor equivalent to the category of $G$-graded finite-dimensional
Hilbert spaces $\Vect_G^\omega$ for some $\omega\in Z^3(G,\TT)$,
where the associator of simple objects $H_g\cong \CC$ is given by 
\begin{align}
  (H_g\otimes H_h)\otimes H_k\xrightarrow{\omega(g,h,k) }
  H_g\otimes (H_h\otimes H_k)\,.
\end{align}
Up to tensor equivalence this only depends on the class 
$[\omega]\in H^3(G,\TT)$, 
see \cite{EtNiOs2010}.
A strict model for $\Vect_G^\omega$ is 
$\langle\alpha_g:g\in G\rangle\subset \End_0(N)$ for a type III
factor $N$ and a $G$-kernel, \ie map $\alpha \colon G\to \Aut(N)$ 
which is a lift of a homomorphism $\chi\colon G\to \Out(N)$ 
having obstruction $[\omega]$, \cf \eg \cite{Su1980,Jo1980,IzKo2002}.

\begin{example}
  For a pointed unitary fusion category with two objects the Frobenius--Schur 
  indicator of the generator is a complete invariant, namely for 
  $\Vect^{\omega_{\pm}}_{\ZZ_2}=\langle\alpha_0=\id,\alpha_1\rangle$ with
  $H^3(\ZZ_2,\TT)=\{[\omega_+]=0,[\omega_-]\}\cong \ZZ_2$ we have
  $\nu_{\alpha_1}=\pm 1$, respectively.
\end{example}

Let $\cC$ be a pointed unitary modular tensor category.
Because of the braiding $G=\Irr(\cC)$ is a finite abelian group.
The category $\cC$ is up to braided equivalence characterized by a 
cohomology 
class $[(\omega,c)]\in H^3_{\mathrm{ab}}(G,\TT)$ in the abelian 
Eilenberg--MacLane cohomology.
It is a fact that $[(\omega,c)]$ is determined by the quadratic form 
$q(a)=c(a,a)$, which equals %
the twist, see \eg \cite{EtGeNiOs2015}.
We use the convention $q([\rho])1_{\rho\otimes\rho}=\varepsilon(\rho,\rho)$ 
for $[\rho]\in \Irr(\cC)$.

A map $b\colon G\times G \to \TT$ with $\bii gh=\bii hg$
and $\bii{g+h}k=\bii gk\bii hk$ and $\bii g\slot \equiv 1$ if and only if $g=0$, is called a \myemph{non-degenerate symmetric 
bicharacter}. 

Let us see $G$ as an additive group and let 
$\partial q(a,b)=q(a)q(b)q(a+b)^{-1}$. 
Then $\partial q$ is a non-degenerate symmetric 
bicharacter. 
In this case, we call the quadratic form $q$ non-degenerate.
Conversely, we call a map $q\colon G\to \TT$ with $q(na)=q(a)^{n^2}$
for every $n\in \ZZ$, such that $\partial q$ is a non-degenerate symmetric 
bicharacter a \myemph{non-degenerate quadratic form} on $G$.

Therefore, every pointed unitary modular tensor category 
is characterized by the pair $(G,q)$ and we denote such a pointed 
unitary modular tensor category by $\cC(G,q)$.

A pair $(G,q)$ of a finite abelian group $G$ and a non-degenerate quadratic 
form $q$ on $G$ is called a \myemph{metric group}.
We say two metric groups $(G,q)$ and $(G',q')$ are equivalent if there is an 
isomorphism $\phi\colon (G,q)\to(G',q')$, \ie an isomorphism 
$\phi\colon G\to G'$, such that $q=q'\circ \phi$.
In this case, we write $(G,q)\sim(G',q')$.
We note that $\cC(G,q)$ and $\cC(G',q')$ are braided equivalent if and only 
if $(G,q)\sim(G',q')$.
For $\{(G_i,q_i):i=1,2\}$ two metric groups we define their direct sum to be 
the metric group $(G_1,q_1)\oplus (G_2,q_2)=(G_1\oplus G_2,q_1\oplus q_2)$ 
with $(q_1\oplus q_2)(g_1,g_2)=q_1(g_1)q_2(g_2)$.
\begin{example}
  The UMTC $\umtc{\SU(n+1)_1}$ is pointed and braided equivalent to
  $\cC(\ZZ_{n+1},q)$ with $q(x)=\exp(\frac{\pi\ima n}{n+1} x^2)$.
  In particular, for $\umtc{\SU(2)_1}=\langle\alpha\rangle\cong \cC(\ZZ_2,q)$
  with $[\alpha^2]=[\id]$ we have $\nu_\alpha=-1$, since 
  $q(x)=\exp(\pi\ima x^2/2)$ does not come from a bicharacter and using
  {\cite[Lemma 4.4]{LiNg2014}}.
\end{example}

The modular data of $\cC(G,q)$
is given by
\begin{align}
  S_{g,h}&=\frac{1}{\sqrt{|G|}}b(g,h)\,, & %
  T_{g,h}&=\delta_{g,h} \e^{-\pi\ima c_\mathrm{top}/12}q(a)%
  \,,
\end{align}
where the \emph{topological central charge} 
$c_\mathrm{top} \equiv c_\mathrm{top} \pmod 8$ is determined via a 
\emph{Gauß sum}:
\begin{align}
  \label{eq:GaussSum}
  \e^{\pi \ima c_\mathrm{top}/4}&=\frac1{\sqrt{|G|}}\sum_{g\in G}  q(g)\,.
\end{align}
This modular data $(S,T)=(S_{(G,q)},T_{(G,q)})$ is sometimes called the 
\myemph{Weil representation} associated with $(G,q)$.

\subsection{Bicharacters and odd rank pointed unitary modular tensor categories}
Let $G$ be a finite abelian group of odd order and 
$\bislotslot\colon G\times G \to \TT$ a
non-degenerate symmetric bicharacter on $G$. 
Define $q\colon G\to \TT$ by $q(g)=\bi gg^{-1}$ and 
$\partial q\colon G\times G \to \TT$ by
$\partial q(g,h)=q(g)q(h)q(g+h)^{-1}$ as above.
Then $\partial q(g,h)=\bi gh^2$ and $\partial q$ 
is a non-degenerate symmetric bicharacter, since $|G|$ is odd.
In particular, $q$ is a non-degenerate quadratic form on $G$.
Again since $|G|$ is odd $\bislotslot$ is determined by $q$. 

Conversely, every non-degenerate quadratic form $q$ 
gives a non-degenerate symmetric bicharacter $\bislotslot$ with $q(g)=\bi gg ^{-1}$.
Namely, $\bi gh:=\partial g(g,h)^{\frac12(\Exp(G)+1)}$ 
where $\Exp(G)$ is the exponent of $G$. 
Since $G$ is odd we can define  
$g'=\frac12(\Exp(G)+1)\cdot g$ with  $g=2g'$ for every $g\in G$. 
We note that in general $q$ must take values in  $\TT_{2\Exp(G)}$, where $\TT_n=\{z\in \TT: z^n=1\}$, but
since $|G|$ is odd we have $q(g)=q(g')^4$ and $q$ actually takes values in $\TT_{\Exp(G)}$.
Then $q(g)^{-1}=\partial q(g,g')=\partial q(g,g)^{\frac12(\Exp(G)+1)}$ and therefore $\bi gg = q(g)^{-1}$. 
Finally, $\bislotslot$ is a symmetric bicharacter and non-degenerate since 
$\partial q$ is non-degenerate.

The following lemma will be useful.
\begin{lem}%
  \label{lem:TrivialCocycle}
  Let $\cC$ be a pointed UMTC of odd rank, \ie $\cC$ braided equivalent to 
  $\cC(G,q)$ for an odd metric group $(G,q)$. Then 
  the corresponding $[\omega]\in H^3(G,\TT)$ is trivial. 
  In particular, $\cC$ is tensor equivalent to $\Vect_G$.
\end{lem}
\begin{proof}
  As above there is a symmetric bicharacter $\bislotslot$ with $q(g)=\bi g g^{-1}$ and the cocycle must therefore be trivial by 
  {\cite[Lemma 4.4]{LiNg2014}}.
\end{proof}

\subsection{Q-systems}
\label{ssec:QSystems}

Let $\cF$ be a unitary fusion category. 
An (irreducible) Q-system $\Theta=(\theta,w,x)$ in $\cF$ is a triple
$\theta\in\cF$ with $\dim\Hom(\tu,\theta)=1$, and isometries
$w\in\Hom(\tu,\theta)$ and $x\in\Hom(\theta,\theta\otimes \theta)$, such that
$(x\otimes \id_\theta)\circ x=(\id_\theta\otimes x)\circ x$ and 
$x^\ast \circ (w\otimes \id_\theta) = x^\ast\circ (\id_\theta\otimes w)
  =\delta^{-1} \id_\theta$ where $\delta=\sqrt{d\theta}$. 
In other words, $\theta$ is an algebra object with unit $e=\delta^{-1/2} w$ and
associative multiplication $\mu=\delta^{-1/2} x^\ast$.
A Q-system $\Theta$ in a braided fusion category $\cC$ is called 
\myemph{commutative} if $x=\varepsilon(\theta,\theta)\circ x$, see
\cite{BiKaLoRe2014-2,BiKaLo2014}.
If $\Theta$ is a Q-system in a unitary fusion category $\cF$, then
$\cF$ and $\bim \Theta\cF\Theta$ are 
\myemph{(weakly monoidally) Morita equivalent} and 
$Z(\cF)$ is braided equivalent to 
$Z(\bim\Theta\cF\Theta)$ \cite{Mg2003,Mg2003II}.

We denote by $\cF_\Theta$ the category of right $\Theta$-modules, 
see \cite{BiKaLoRe2014-2,BiKaLo2014}.
If $\cC$ is braided and $\Theta$ is commutative, then $\cC_\Theta$ has the 
structure of a fusion category and the category of \myemph{local} modules
$\cC_\Theta^0$ has the structure of a braided fusion category.
Let us a assume that $\cC$ is a modular tensor category, then 
$\bim \Theta \cC\Theta=\langle \bim[+] \Theta \cC\Theta, 
\bim[-] \Theta \cC\Theta\rangle$  \cite[Theorem 5.10]{BcEvKa1999}, where
$\bim[\pm]\Theta\cC\Theta=\langle\alpha_\Theta^\pm(\cC)\rangle$
are categories obtained by $\alpha$-induction. 
With our convention $\cC_\Theta$ is tensor equivalent to 
$\bim[+]\Theta\cC\Theta$ and
$\bim[0]\Theta \cC\Theta:= 
  \bim[+] \Theta \cC\Theta\cap \bim[-] \Theta \cC\Theta$ has a natural braiding
and is braided equivalent to $\cC_\Theta^0$.

We say that an injective tensor functor 
$\iota\colon \cD\hookrightarrow\cF$ from a braided
unitary fusion category $\cD$ to a unitary fusion category $\cF$ is 
\myemph{central}, if there is a braided functor 
$\tilde\iota \colon \cD\to Z(\cF)$, such that $\iota=F\circ\tilde\iota$, 
\ie the following diagram commutes
$$
\begin{tikzcd}
    & Z(\cF) \arrow[d,"F",two heads]  \\
   \cD\arrow[ru,"\tilde \iota",dashed]\arrow[r,"\iota",hookrightarrow]  
   & \cF 
  \end{tikzcd}\,.
$$
Here $F\colon Z(\cF)\to \cF$ denotes the forgetful functor. 

A braided unitary fusion category is called \myemph{symmetric}
if $\cC'\cap \cC=\cC$.
Let $G$ be a finite group and consider the symmetric unitary fusion category
$\Rep(G)$ of finite dimensional unitary representations of $G$ with the usual
tensor product.
The regular representation defines a canonical Q-system $\Theta_G\in \Rep(G)$.
If $\cF$ has a \myemph{Tannakian subcategory}, \ie 
if there is a central functor $\iota\colon\Rep(G)\to \cF$, then we denote 
$\cF_G:=\cF_{\iota(\Theta_G)}$, which is also called the 
\myemph{de-equivariantization} of $\cF$.

If $\cC$ is a unitary modular tensor category and $\iota\colon\Rep(G)\to\cC$ braided, then $\cC_G$ is a \emph{$G$-crossed braided extension} of 
the unitary modular tensor category $\cC_G^0$ which we take for the 
purpose of this article as a definition,
\cf M\"uger's characterization \cite[Appendix 5, Thm 4.1]{Tu2010}.
This characterization also says that every 
\myemph{$G$-crossed braided category} is of the form $\cC_G$.
We refer to \cite{Mg2005} and \cite[Appendix 5]{Tu2010} for the 
definition of a \myemph{$G$-crossed braided category}.
We note that in this case the inclusion has the structure of a 
central functor $\iota\colon\rev{(\cC^0_G)}\hookrightarrow \cC_G$.
Also the following converse is true, which is probably well-known to experts.
\begin{prop}
Let $\cD$ be a unitary modular tensor category and $\cF=\bigoplus_g \cF_g$ a 
(faithfully) $G$-graded extension of $\cF_e=\cD$ together with a 
central structure  $\iota\colon \rev{\cD}\to \cF$ on the 
canonical inclusion functor. 
Then $\cF$ is a $G$-crossed braided extension of $\cD$.
\end{prop}
\begin{proof}
By \cite[Prop.\ 5.12]{Bi2016} there is a commutative Q-system $\Theta$ 
in $\cC:=\iota(\rev{\cD})'\cap Z(\cF)$, such that $\cC_\Theta$ is equivalent 
to $\cF$ and $\cC_\Theta^0$ is braided equivalent to $\cC$.
Similarly to the proof of \cite[Prop.\ 5.17]{Bi2016} it 
follows that $\Theta$ is the Q-system of a group subfactor $M^G\subset M$
and since $\Theta$ is commutative we have a braided functor 
$\iota\colon \Rep(G)\to \cC$, such that $\cC_G^0$ is braided equivalent to 
$\cD$ and $\cC_G\cong \cF$.
\end{proof}

\subsection{Even lattices}
A \myemph{(positive) even lattice} is a finitely generated free abelian group 
$L$ with a positive definite inner product 
$\bislotslot\colon L\times L \to \RR$, such that $\langle x,x\rangle \in 2\NN$
for all $x\in L\setminus \{0\}$.
Then $V_L=L\otimes_\ZZ\RR$ is an euclidean space, with inner product 
$\bislotslot$ defined by bilinear continuation. 
The number of generators, or equivalently $\dim(V_L)$, is called the 
\myemph{rank} of $L$.
The \myemph{dual lattice} is 
$L^\ast=\{x\in V_L: \bi x L \subset \ZZ\}\cong \Hom(L,\TT)$. 
Since $L$ is even, we have $L\subset L^\ast$ and $G_L=L^\ast/L$ is a finite 
abelian group with non-degenerate quadratic form $q_L\colon G_L\to \TT$ given by
$q_L(x)=\exp(\pi \ima \bi xx)$. 
Thus every even lattice $L$ gives a metric group $(G_L,q_L)$, which is called
the \myemph{discriminant group} of $L$.
If $H\leq G_L$ is an isotropic subgroup, \ie $q_L\restriction H\equiv 1$, we get
a new even lattice $L\oplus H=\{x+h : x\in L, [h]\in H\}$.
This gives a one-to-one correspondence between overlattices $M\supset L$ and 
isotropic subgroups $H\leq G_L$ given by $H\mapsto M:=L\oplus H\supset L$ and 
$M\mapsto H:= M/L\leq G_L$. 
A lattice $L$ is called \myemph{self-dual} if $L^\ast/L$ is trivial.
If $L$ and $M$ are even lattices then $L\lattimes M:=L\times M$ is an even 
lattice with $G_{L\lattimes M}=G_L\times G_M$ and 
$q_{L\lattimes M}(x, y)=q_L(x)q_M(y)$.
We say that an even lattice $\bar L$ is a \myemph{mirror} of $L$ if there is an
isomorphism $\phi\colon(G_L,q_L) \to (G_{\bar L},\bar q_{\bar L})$. 
Then we get a self-dual lattice $\Gamma=L\bar L\oplus \Delta^\phi(G_L)$ via the 
isotropic ``diagonal'' subgroup 
$\Delta^\phi(G_L)=\{(x,\phi(x))\in  G_L\times G_{\bar L}  : x\in G_L\}\leq
G_L\times G_{\bar L}$. 

We will often refer to the $A_n, E_{6,7,8}$ root lattices whose 
discriminant groups are listed in Table \ref{tab:lattices}. 
We note that $E_7$ and $E_6$ are mirrors of $A_1$ and $A_2$, respectively. 
There is a unique overlattice 
$A_1 \lattimes E_7\oplus \ZZ_2\supset A_1\lattimes E_7$ and two 
overlattices $A_2\lattimes E_6\oplus \ZZ_3\supset A_2\lattimes E_6$ 
which are both isomorphic to $E_8$.
\begin{table}[h]
  \begin{tabular}{c|c|c}
    $L$ & $G_L$ & $q_L$\\\hline
    $A_{n}$  & $\ZZ_{n+1}$ & $\exp{\frac{\pi \ima x^2n}{n+1}}$\\
    $E_6$  & $\ZZ_3$ & $\exp{\frac{\pi \ima 4 x^2}{3}}$\\
    $E_7$  & $\ZZ_3$ & $\exp{\frac{\pi \ima 3 x^2}{2}}$\\
    $E_8$  & $\ZZ_1$ & $1$
  \end{tabular}
  \caption{Discriminant groups for $A_n, E_{6,7,8}$  root lattices}
  \label{tab:lattices}
\end{table}

\subsection{Conformal nets}
\label{subsec:CN}
We denote by $\cI$ the set of proper (\ie open, non-empty, and non-dense) 
intervals $I\subset \Sc$ on the circle and by $I'=\Sc\setminus \overline{I}$. 
Let us denote the group of orientation preserving diffeomorphisms of the circle
$\Sc$ by $\Diff_+(\Sc)$. 
We note that the M\"obius group $\Mob$ is naturally a subgroup of 
$\Diff_+(\Sc)$.
By a \myemph{(local) conformal net} $\A$, we mean a 
local M\"obius covariant net on the
circle, which is diffeomorphism covariant. 
Although, we do not use diffeomorphism covariance, all nets we consider have 
this property.

More precisely, a conformal net associates with every interval $I\in\cI$ a von 
Neumann algebra $\A(I)\subset \B(\Hil)$ on a fixed Hilbert space $\Hil=\Hil_\A$,
such that the following properties hold:
\begin{enumerate}[{\bf A.}]
  \item \myemph{Isotony.} $I_1\subset I_2$ implies $\A(I_1)\subset \A(I_2)$.
  \item \myemph{Locality.} $I_1  \cap I_2 = \emptyset$ implies 
    $[\A(I_1),\A(I_2)]=\{0\}$.
  \item \myemph{Möbius covariance.} There is a unitary representation $U$ of 
    $\Mob$ on $\Hil$, such that $U(g)\A(I)U(g)^\ast = \A(gI)$.
  \item \myemph{Positivity of energy.} $U$ is a positive energy representation,
    \ie the generator $L_0$ (conformal Hamiltonian) of the rotation subgroup 
    $U(z\mapsto \e^{\ima \theta}z)=\e^{\ima \theta L_0}$ has positive spectrum.
  \item \myemph{Vacuum.} There is a (up to phase) unique rotation invariant 
    unit vector $\Omega \in \Hil$ which is cyclic for the von Neumann algebra 
    $\A:=\bigvee_{I\in\cI} \A(I)$.
  \item \myemph{Diffeomorphism covariance.} There is a projective unitary 
    representation $U$ of $\Diff_+(\Sc)$ extending the representation $U$ of
    $\Mob$, such that for all $I\in\cI$ 
    \begin{align*}
      U(g)\A(I)U(g)^\ast &= \A(gI)\,, & g\in \Diff_+(\Sc)\,, \\
      U(g)xU(g)^\ast &= x\,, &x\in \A(I'),~ g\in \Diff_+(I')
      \,.
    \end{align*}
    where $\Diff_+(I)=\{ g\in\Diff_+(\Sc) : g\restriction I'=\id_{I'}\}$.
\end{enumerate}
Let $I\in\cI$. 
By Haag duality $\A(I')=\A(I)'$ \cite{GaFr1993,BrGuLo1993} we have 
$U(g)\in\A(I)$ for each $g\in \Diff(I)$ and 
$\Vir_\A(I):=\{U(g): g\in\Diff(I)\}''\subset\A(I)$ defines a subnet of $\A$, 
the so-called \myemph{Virasoro net} associated with $\A$.
The positive energy representation of $\Diff_+(\Sc)$ restricted to 
$\Hil_{\Vir_\A}=\overline{\bigvee_I\Vir_\A(I)\Omega}$ is an irreducible 
positive energy representation of $\Diff_+(\Sc)$ with an $\Mob$ invariant 
vector $\Omega$ (see \cite{Ca2004,CaWe2005}). 
Such representations are completely classified by the \myemph{central charge}
$c$ and $\Vir_\A\cong\Vir_c$ for some unique $c>0$.

A conformal net $\A$ is called \myemph{completely rational} if it 
\begin{enumerate}[{\bf A.}]
  \setcounter{enumi}{6}
  \item fulfills the \myemph{split property}, \ie 
    for $I_0,I\in \cI$ with $\overline{I_0}\subset I$ the inclusion 
    $\A(I_0) \subset \A(I)$ is a split inclusion, namely there exists an 
    intermediate type I factor $M$, such that $\A(I_0) \subset M \subset \A(I)$.
  \item is \myemph{strongly additive}, \ie for $I_1,I_2 \in \cI$ two adjacent 
    intervals obtained by removing a single point from an interval $I\in\cI$
    the equality $\A(I_1) \vee \A(I_2) =\A(I)$ holds.
  \item for $I_1,I_3 \in \cI$ two intervals with disjoint closure and 
    $I_2,I_4\in\cI$  the two components of $(I_1\cup I_3)'$, the 
    \myemph{$\mu$-index} of $\A$
    \begin{equation*}
      \mu(\A):= [(\A(I_2) \vee \A(I_4))': \A(I_1)\vee \A(I_3) ]
    \end{equation*}
    (which does not depend on the intervals $I_i$) is finite.
\end{enumerate}

\begin{rmk}
  It was recently shown, that diffeomorphism covariance implies the split 
  property \cite{MoTaWe2016}. 
  Further, diffeomorphism covariance, split property and finite $\mu$-index 
  implies strong additivity \cite{LoXu2004}.
  Thus finite $\mu$-index is equivalent to completely rationality if we assume 
  diffeomorphism covariance.
\end{rmk}

A \myemph{representation} $\pi$ of a strongly additive net $\A$ 
is a family of (unital) representations
$\pi=\{\pi_I\colon\A(I)\to \B(\Hil_\pi)\}_{I\in\cI}$ on a common Hilbert space
$\Hil_\pi$ which are compatible, \ie  $\pi_J\restriction \A(I) =\pi_I$ for 
$I\subset J$.
Every representation $\pi$ with $\Hil_\pi$ separable---for every choice of an interval $I_0\in\cI$---turns out to be equivalent to a representation 
\myemph{localized} in $I_0$, \ie $\rho$ on $\Hil$, such that 
$\rho_J=\id_{\A(J)}$ for $J\cap I_0=\emptyset$. 
Then Haag duality implies that $\rho_{I}$ is an endomorphism of $\A(I)$ for 
every $I \in \cI$ with $I\supset I_0$, which we also denote by $\rho$.
The \myemph{statistical dimension} of a representation $\rho$ localized in $I$ 
is given by the square root of the Jones index
of the Jones--Wassermann subfactor $d\rho=[\A(I):\rho(\A(I))]^{\frac12}$.
By \cite[Cor.\ 39]{KaLoMg2001} every representation of a completely rational 
conformal net is a direct sum of representations with finite statistical 
dimension. 
For convenience we will restrict to representations with finite statistical dimension, thus every representation is a finite direct sum of irreducible 
representations and we obtain a semisimple category.

Thus we can realize the category of representations of $\A$ with finite 
statistical dimension which are localized in $I$
inside the rigid C$^\ast$-tensor category of endomorphisms $\End_0(N)$
of the type III factor $N=\A(I)$ and the embedding turns out to be full and 
replete. 
We denote this category by $\Rep^I(\A)$.
In particular, this gives the representations of $\A$ the structure of a tensor category \cite{DoHaRo1971}. 
It has a natural \myemph{braiding}, which is completely fixed by asking that if
$\rho$ is localized in $I_1$ and $\sigma$ in $I_2$ where $I_1$ is left of $I_2$
inside $I$, then $\varepsilon(\rho,\sigma)=1$ \cite{FrReSc1989}. 
Let $\A$ be completely rational conformal net, then by \cite{KaLoMg2001} 
$\Rep^I(\A)$ is a UMTC and $\mu_\A=\Dim(\Rep^I(\A))$.
A completely rational conformal net is called \myemph{holomorphic}, if 
$\Rep^I(\A)$ is trivial, \ie equivalent to the category 
finite dimensional Hilbert spaces $\Hilb$, or equivalently $\mu_\A=1$.

We write $\A\subset \cB$ or $\cB\supset \A$ if there is a representation 
$\pi=\{\pi_I\colon\A(I)\to\cB(I)\subset\B(\Hil_\cB)\}$ of $\A$ on 
$\Hil_\cB$ and an isometry $V\colon \Hil_\A\to \Hil_\cB$ with 
$V\Omega_\A=\Omega_\cB$ and $VU_\A(g)=U_\cB(g)V$.
We furthermore ask that $Va=\pi_I(a)V$ for $I\in\cI$, $a\in\A(I)$. 
Define $p$ to be the orthogonal projection onto 
$\Hil_{\A_0}=\overline{\pi_I(\A(I))\Omega}$.  
Then $pV$ is a unitary equivalence of the nets $\A$ on $\Hil_\A$ and $\A_0$ 
defined by $\A_0(I)=\pi_I(\A(I))p$ on $\Hil_{\A_0}$.

The inclusion $\A\subset \cB$ is called finite index if the Jones index
$[\cB(I):\A(I)]$ is finite. 
In this case, $\A$ is completely rational if and only if $\cB$ is completely
rational \cite{Lo2003}. 
The conformal net $\cB$ is characterized by a commutative Q-system
$\Theta$ in $\Rep(\A)$ \cite{LoRe1995} and $\Rep(\A)$ is braided equivalent to 
$\Rep(\cB)^0_\Theta$ \cite{BiKaLo2014}.

Let $\cB$ be a completely rational conformal net, we note that two
extensions $\A\supset \cB$ and $\tilde \A \supset\cB$ are \myemph{isomorphic},
if and only if they have equivalent Q-systems in $\Rep(\cB)$, which 
can be taken as a definition for the purpose of this paper.

\begin{defi}[\cite{KaLo2005}]
  \label{defi:ModularNet}
  A diffeomorphism covariant completely rational net $\A$ with central charge 
  $c$ is called \myemph{modular} if for 
  \begin{align}
    \chi_\rho(\tau)&=\tr\left(\e^{2\pi\ima\tau (L^\rho_0-c/24)}\right)
  \end{align}
  we have a  representation of $\SL(2,\ZZ)$ with 
  $T^\chi=\diag( \e^{2\pi\ima (L^\rho_0-c/24)})$ and
  \begin{align}
    \chi_\rho(-1/\tau)
      &=\sum_{\nu\in\Irr(\Rep(\A))} S^\chi_{\rho,\nu}\chi_\nu(\tau)\,,\\
    \chi_\rho(\tau+1)
      &=\sum_{\nu\in\Irr(\Rep(\A))} T^\chi_{\rho,\nu}\chi_\nu(\tau)\,,
  \end{align}
  such that $(S^\chi,T^\chi)$ coincides with the (categorical) modular data 
  $(S,T)$ of $\Rep(\A)$, 
  \ie $S=S^\chi$ and $T=\omega_3 T^\chi$ for some third root of unity 
  $\omega_3$.
\end{defi}
By the spin-statistic theorem \cite{GuLo1995} the requirement on $T$ is $c_\mathrm{top}(\Rep(\A))\equiv c\pmod 8$.
Conformal nets with $c<1$ and the nets $\A_{\SU(N)_k}$ are modular 
\cite{Xu2000,Xu2001}, \cf also \cite{KaLo2005}.

\subsection{Orbifold theories}
Fixed points of conformal nets under group actions, so-called orbifolds, 
were studied in \cite{Xu2000-2,Mg2005}.
An automorphism of a conformal net $\A$ is a compatible family 
$\{\alpha_I\in\Aut(\A(I))\}$ of automorphisms which preserve the vacuum, 
\ie $(\Omega,\alpha_I(a)\Omega)= (\Omega,a\Omega)$ for all $a\in\A(I)$. 
The group of all automorphisms of $\A$ is denoted by $\Aut(\A)$.

Let $G\leq \Aut(\A)$ be a finite group, then the fixed point net 
$\A^G\subset \A$ given by
$\A^G(I)=\{a\in\A(I):\alpha_I(a)=a \text{ for all }\alpha\in G\}$ 
is a finite index subnet with index $[\A :\A^G]=|G|$. 
The Q-system $\Theta$ in $\Rep(\A^G)$ giving $\A\supset \A^G$ is the regular 
representation of $G$ and $\Rep(\A)=\Rep(\A^G)_G^0$.

We identify $\Sc\setminus\{-1\}$ with $\RR$ and fix $I\Subset \RR$.
For $\alpha\in \Aut(\A)$, we say $\pi$ is an $\alpha$-representation
of $\A$, if $\pi$ is a representation of $\A$ on $\RR$, 
such that $\pi_{I_-}=\id_{\A(I_-)}$ and $\pi_{I_+}=\alpha_{I_+}$
for $I_-<I<I_+$.
We define ${\tRep G}(\A)={\tRep G}^I(\A)$ to be the category of 
representations of $\A$ on $\RR$ which are finite direct sums  
of $\alpha_g$-representations for $\alpha_g\in G$.

The category generated from $\alpha^+$-induction of $\Rep^I(\A^G)$ 
for $\A(I)^G \subset \A(I)$ is equivalent to $\tRep{G}^I(\A)$ which implies that 
${\tRep G}(\A)$ is tensor equivalent to $\Rep(\A^G)_G$.
In particular, the category ${\tRep G}^I(\A)$ is a 
$G$-crossed braided extension of $\Rep^I(\A)$.
Furthermore, $\Rep^I(\A^G)$ is braided equivalent to $(\tRep G^I(\A))^G$.
We refer to \cite{Mg2005} for more details.

\subsection{Generalized orbifolds}
Generalized orbifolds in conformal nets were introduced by the author 
in \cite{Bi2016}. 
A \myemph{(finite) hypergroup} $K$ is a finite set, 
which is the basis of a (finite-dimensional) C${}^\ast$-algebra $\CC K$,
such that the identity $1\in K$, the set $K$ is closed under adjoints 
(\ie $K^\ast= K$),
the multiplication restricts to a map 
$m\colon K\times K \to \Conv(K)=\{\sum_{k\in K}\lambda_k k: 
\sum_k\lambda_k=1 \text{ and } \lambda_k\geq 0\}$, and we 
 have the following antipode law:
 $1\prec k\cdot \ell$ for some $k,\ell \in K$ if and only if 
$k=\ell^\ast$.
Here, for $\ell\in K$ we write 
$\ell \prec \sum_k\lambda_k k\in \Conv(K)$ if $\lambda_\ell >0$.

A finite group $G$ is a hypergroup with the usual multiplication and
$g^\ast = g^{-1}$.
Conversely, a hypergroup, such that the multiplication $m\colon K\times K\to\Conv(K)$ takes values in $K$ is a finite group. 
There is an obvious notion of a subhypergroup $L\leq K$ and 
the double quotient $K\CS L$ is again a hypergroup.
If $G$ is a finite group we have the Tambara--Yamagami hypergroup
$K=G\cup \{\rho\}$ with $\rho^\ast=\rho=g\rho=\rho g$ for all $g\in G$ 
and $\rho^2=\frac{1}{|G|}\sum_g g$.
More generally, if $\cF$ is a unitary fusion category, then we have
the associated fusion hypergroup given by the renormalized basis
$K_\cF=\{d\rho^{-1}[\rho] :[\rho]\in \Irr(\cF)\}$ of the 
complexified fusion ring or fusion algebra $K_0(\cF)\otimes_\ZZ\CC$.

Let $\A$ be a conformal net. 
A \myemph{quantum operation} on $\A$ is a compatible family 
$\phi=\{\phi_I\colon \A(I)\to\A(I)\}$ of 
extremal normal unital completely positive maps which are vacuum preserving
and have an adjoint $\phi_I^\#$ with $(a\Omega,\phi_I(b)\Omega)=
(\phi_I^\#(a)\Omega,b\Omega)$ for all $a,b\in\A(I)$. 
The set of all quantum operations on $\A$ is denoted by $\QuOp(\A)$.
We note that $\Aut(\A)\subset\QuOp(\A)$.
One of the main result of \cite{Bi2016} can be restated as follows.
There is a one-to-one correspondence between finite index subnets 
$\cB\subset \A$ and finite hypergroups $Q\leq \QuOp(\A)$, where by 
$Q\leq \QuOp(\A)$ we mean a finite subset $Q$ which forms a hypergroup with multiplication
given by composition and adjoints given by $\phi^\ast= \phi^\#$.
The correspondence is given by $Q\mapsto \A^Q\subset\A$
where $\A^Q$ is the fixed point net 
$\A^Q(I)=\{a\in\A(I):\phi_I(a)=a \text{ for all }\phi\in Q\}$. 

If $\A$ is completely rational and $\cB\subset \A$ finite index
then the unique $Q\leq \QuOp(\A)$ with $\A^Q=\cB$  is isomorphic to
the double quotient $K_\cF\CS K_{\Rep(\A)}$, 
where $\cF$ is the category generated by
$\alpha^+$-induction for the inclusion $\cB(I)\subset \A(I)$.
In pure analogy with the group case we denote $\tRep{ Q}^I(\A):=\cF$.

Thus for any finite hypergroup $Q\leq \QuOp(\A)$ there is a unitary fusion 
category $\tRep {Q}^I(\A)$ extending $\Rep^I(\A)$, such that 
$Q\cong  K_{\tRep Q(\A)}\CS K_{\Rep(\A)}$. 
Furthermore, $\tRep {Q}^I(\A)$ generalizes  
$\tRep {G}^I(\A)$  in the case $Q=G$ is a finite group.
The inclusion $\Rep(\A)\to \tRep Q(\A)$ has naturally the 
structure of a central functor 
$\rev{\Rep(\A)}\to \tRep Q(\A)$, such that  
$\Rep(\A^Q)$ is braided equivalent to the M\"uger centralizer
$(\rev{\Rep(\A)})'\cap Z(\tRep Q(\A))$
of $\rev{\Rep(\A)}$ in the Drinfel'd center $Z(\tRep Q(\A))$.
In the case that $\A$ is holomorphic, $\tRep{Q}^I(\A)$ is a categorification 
of $Q$, \ie $K_\tRep{Q}^I(\A)=Q$ and $\Rep(\A^Q)$ is braided equivalent to
$Z(\tRep Q(\A))$.

\section{Realization of Tambara--Yamagami categories and their centers}
\subsection{Changing Frobenius--Schur indicators}
We show that starting with a unitary modular tensor category $\cC$ with certain properties, we obtain a new (twisted) unitary modular tensor category 
$\hat \cC$ with the same fusion rules but different Frobenius--Schur indicators.
We call the unitary modular tensor category 
$\cS =\umtc{\SU(2)_1}$ the \myemph{semion category}.
\begin{prop}[Changing the Frobenius--Schur Indicator]
  \label{prop:ChangingFrobeniusSchur}
  Let $\cC=\cC_0\oplus \cC_1$ be a $\ZZ_2$-graded 
  unitary modular tensor category and $\alpha\in\cC_0$ with 
  $\langle\alpha\rangle$ braided equivalent to $\Rep(\ZZ_2)$.
  Let $\cS=\langle\tau\rangle$ the semion category and let $\Theta$ be the $\ZZ_2$ 
  Q-system associated with $[\theta]=[\id]\oplus[ \alpha\boxtimes \tau\boxtimes \bar \tau]$. 
  Then 
  \begin{enumerate}
    \item $\hat \cC:=(\cC\boxtimes\cS\boxtimes \rev{\cS})_{\Theta}^0$ has 
  the same fusion rules as $\cC$, \ie there is a map $\Irr(\cC)\to \Irr(\hat \cC)\colon\rho\mapsto\hat\rho$
  giving an isomorphism of Grothendieck rings.
    \item $\cC\mapsto\hat \cC$ is involutive, \ie
   $\hat{\hat{\cC}}$ is braided equivalent to $\cC$.
     \item $\hat\cC$ is braided equivalent to the subcategory $\langle
  \hat\rho_0:=\rho_0\boxtimes \id\boxtimes \id, \hat\rho_1:=\rho_1\boxtimes \tau \boxtimes \id
  :\rho_i\in\cC_i\rangle$ of $\cC\boxtimes\cS \boxtimes\rev{\cS}$. 
  In particular, objects in $\hat \cC_1$ have opposite Frobenius-Schur 
  indicators compared to the corresponding objects in $\cC_1$, \ie 
  $\nu_{\hat\rho_1}=-\nu_{\rho_1}$ for all (self-dual) $\rho_1\in\cC_1$.
  \end{enumerate}
\end{prop}
\begin{proof}
  We have subcategories 
  $\hat\cC =\langle \rho_0\boxtimes \id\boxtimes \id, \rho_1\boxtimes \tau
  \boxtimes \id :\rho_i\in\cC_i\rangle$  and $\cD=\langle \alpha\boxtimes \id
  \boxtimes\bar\tau, \id\boxtimes \tau\boxtimes \id\rangle$ of $\cC\boxtimes
  \cS \boxtimes \rev{\cS}$. 
  Then it follows that $\cD \cong \cS \boxtimes \rev{\cS}$. 
  By Müger's theorem \cite[Theorem 4.2]{Mg2003-MC} we have 
  $\cC\boxtimes \cS \boxtimes \rev{\cS}\cong \hat \cC\boxtimes \cD$. 
  Then the canonical algebra $\Theta$ in $\cD$ is
  $[\id\boxtimes\id\boxtimes\id] \oplus [\alpha\boxtimes\tau\boxtimes\bar\tau]$
  and  $(\cC\boxtimes \cS \boxtimes \rev{\cS})_\Theta^0\cong \hat \cC$.

  Choosing $\hat\cC \cong \langle \rho_0\boxtimes \id\boxtimes \id, 
    \rho_1\boxtimes \id \boxtimes \bar \tau :\rho_i\in\cC_i\rangle$  
  and $\cD=\langle \alpha\boxtimes \tau\boxtimes \id, 
    \id\boxtimes \id\boxtimes \bar \tau\rangle$ gives braided equivalent tensor 
  categories.
  Therefore, we have $\hat{\hat{\cC}}\cong\langle \rho_0\boxtimes\id\boxtimes\id
  ,\rho_1\boxtimes\tau\boxtimes \bar\tau : \rho_i\in\cC_i\rangle \cong \cC$.
\end{proof}
\begin{prop}
  \label{prop:ChangeFS}
  Let $\A$ be a completely rational net with $\cC:=\Rep(\A)$ fulfilling 
  the assumption of Proposition \ref{prop:ChangingFrobeniusSchur}.
  Then there is a completely rational net $\hat \A$
  with $\Rep(\hat\A)\cong \hat \cC$.

  $\A$ is a $\ZZ_2$-orbifold $\A=\A_L^{\ZZ_2}$
  of the conformal net $\A_L$ associated with an even lattice $L$,
  there is a proper $\ZZ_2$-action on
  $\A_{L \lattimes E_8}$, such that we can choose $\hat \A=\A_{L\lattimes
  E_8}^{\ZZ_2}$.
\end{prop}
\begin{proof}
  Let $\A_{A_1\lattimes E_7}$ be the conformal net associated with the lattice 
  $A_1\lattimes E_7$, then $\Rep(\A_{A_1\lattimes E_7})$ is braided equivalent 
  to $\cS\boxtimes \rev{\cS}$.
  Consider, the $\ZZ_2$-simple current extension 
  $\hat \A=(\A \otimes \A_{A_1\lattimes E_7})\rtimes \ZZ_2$
  w.r.t.\ $\alpha\otimes \alpha_{1,1}$, where $\alpha_{1,1}$ correspond to
  $\tau\boxtimes \bar\tau$. 
  Then it follows directly that $\Rep(\A)$ is braided equivalent to $\hat \cC$.

  Let us consider $\A_L^{\ZZ_2}\otimes \A_{A_1\lattimes E_7}$. Then 
  we can make a $\hat\ZZ_2\times\hat\ZZ_2$-simple current extension giving
  $\A_{L\lattimes E_8}$. 
  We get a $\ZZ_2\times\ZZ_2$ action with 
  \begin{align}
    \A_{L\lattimes E_8}^{\langle (1,0)\rangle}&=\A_L^{\ZZ_2}\otimes \A_{E_8}\\
    \A_{L\lattimes E_8}^{\langle (0,1)\rangle}&=\A_L\otimes \A_{A_1 \lattimes E_7}\\
    \A_{L\lattimes E_8}^{\langle (1,1)\rangle}&=(\A_L^{\ZZ_2}\otimes \A_{A_1\lattimes E_7})\rtimes_{111}\ZZ_2
  \end{align}
  and we can choose $\hat \A =\A_{L\lattimes E_8}^{\langle (1,1)\rangle}$.
\end{proof}

\begin{example} 
  Let $\A_{\SU(2)_k}$ be the loop group net of $\SU(2)$ at level $k$ \cite{Wa,Xu2000},
  then it follows that $\Rep(\A_{\SU(2)_k})$ is braided equivalent to
  $\umtc{\SU(2)_k}$ using the classification \cite{FrKe1993}, \cf 
  \cite{He2017}.
  The simple objects are $\left\{\rho_0,\rho_{\frac12},\ldots,\rho_{\frac k2}\right\}$ with fusion rules
  $$
  [\rho_{ i}]\times [\rho_{j}]=\bigoplus_{\substack{\ell=|i-j|\\i+j+\ell\leq k}}^{i+j}[\rho_{\ell }].
  $$

  The unitary modular tensor category $\cC=\umtc{\SU(2)_{4k}}$ fulfills the 
  assumption of Proposition \ref{prop:ChangingFrobeniusSchur}. 
  Since in $\hat\cC$ the generating object $\hat \rho_{\frac12}$ has trivial 
  Frobenius--Schur indicator, it turns out that $\hat \cC$ 
  is what could be called the \emph{Jones--Kaufmann (modular tensor) category} 
  \cf \cite{Wang2010}.

  Therefore we get a completely rational net 
  $\hat{\A}_{\SU(2)_{4k}}=(\A_{\SU(2)_{4k}}\otimes 
    \A_{A_1\lattimes E_7})\rtimes\ZZ_2$
  realizing the Jones--Kaufmann category $\hat \cC$.
  We can replace $\A_{\SU(2)_{4k}}$ by the net constructed in \cite{Bi2015}
  which realizes  $\rev{(\umtc{\SU_{4k}})}$, to obtain a
  conformal net realizing $\rev{\hat \cC}$.

  We note that  $\cC \mapsto \hat \cC$ corresponds to the 
  $\ZZ_2$-twist of the UMTC $\umtc{\SU(2)_{4k}}$.
  Namely, Kazhdan and Wenzl showed in \cite{KaWe1993} that fusion categories 
  with $\SU(N)_k$ fusion rules are $\umtc{\SU(N)_k}$ possibly twisted
  by an element of $\ZZ_N$.
\end{example}

\subsection{Changing $H^3$ in $G$-crossed braided categories of orbifold nets}
\label{sec:ChangeH3}
We briefly generalize the result from $\ZZ_2$ to an arbitrary finite group $G$. 
We note that the Frobenius--Schur inidcator in Prop.\ \ref{prop:ChangeFS} comes from a class in $H^3(\ZZ_2,\TT)$ which classifies the $\ZZ_2$-extension 
$\tRep{\ZZ_2}(\A\rtimes_\alpha \ZZ_2)$ of $\Rep(\A\rtimes_\alpha\ZZ_2)$.

In \cite{EtNiOs2010} it is shown that 
$G$-crossed braided extensions $\cF$ of $\cC$ are parametrized by an group homorphism $c\colon G \to \Pic(\cC)$,
$M\in H^2(G,\pi_2)$ and a $t_\cF \in H^3(G,\CC^{\times})$, such that certain obstructions 
$o_3(c)\in H^3(G,\pi_2)$ and $o_4(c,M)\in H^4(G,\CC^{\times})$ vanish.

Let $\A$ be a completely rational net $G\leq \Aut(\A)$, then $\tRep G(\A)$ is a $G$-crossed braided category
which is a $G$-graded extension of $\Rep(\A)$. 
The classification of $G$-extensions involves a  $[\varphi]\in H^3(G,\TT)$ which we can twist as follows.
Assume $\cB$ is a holomorphic net with an action of $G$ 
such that $\tRep G(\cB)\cong \Vect_G^\omega$ for some $[\omega]\in H^3(G,\TT)$. Then we can take the diagonal action 
of $G$ on $\A\otimes \cB$ and we get that $\tRep G(\A\otimes \cB)$ 
which has the same fusion rules as $\tRep G(\A)$ but gives the class $[\varphi+\omega]\in H^3(G,\TT)$. 
Evans and Gannon announced that for every finite group $G$ and every 
$[\omega]\in H^3(G,\TT)$ there is a conformal net $\A_{G,\omega}$ with 
$\Rep(\A_{G,\omega})$ braided equivalent to $Z(\Vect_G^\omega)$,
thus from the Lagrangian Q-system coming from the induction functor 
$I\colon \Vect_G^\omega \to Z(\Vect_G^\omega)$ 
we get a holomorphic extension $\cB_{G,\omega}=\A_{G,\omega}\rtimes \Rep(G)$ 
with $\tRep{G}(\A)\cong \Vect_G^\omega$.

Thus we have proven 
\begin{prop} 
  Let $\cC$ be a UMTC and $\cF$ a $G$-crossed braided extension of $\cC$ with 
  $t_\cF\in H^3(G,\TT)$.
  If there is a completely rational net $\A$ realizing $\cF^G$, then there is a
  completely rational net realizing $\tilde \cF^G$ for every 
  $[\tilde \varphi]\in H^3(G,\TT)$, where $\tilde \cF$ is the 
  $G$-crossed braided extension
  of $\cC$  similar to $\cF$, but with $t_{\tilde \cF}=[\tilde\varphi]$ 
  instead of $\varphi$.
\end{prop}
This seems to be a step towards realizing Drinfel'd centers of nilpotent fusion categories by conformal nets.
\subsection{Realization of pointed unitary modular tensor categories}

Let $\cC$ be a pointed UMTC, then $\cC$ is braided equivalent to $\cC(G,q)$.
It follows from \cite{Ni1979} (see \cite{MO233580}) that there 
is an even positive lattice $L=(L,\bislotslot)$, such that 
$(G,q)$ is equivalent to the discriminant form $(G_L,q_L)$ with $G_L=L^\ast/L$ 
and $q_L(\ell+L)=\exp(\pi\ima\langle \ell,\ell\rangle)$.

Let us consider the conformal net $\A_L$ associated with $L$, see 
\cite{DoXu2006,Bi2012}.
From \cite{DoXu2006} it follows that $\Rep(\A_L)$ has $G_L$ fusion rules, 
namely the sectors are $\{[\rho_{m+L}]:m+L\in L^\ast/L\}$ with fusion rules
$[\rho_{m+L}][\rho_{n+L}]=[\rho_{m+n+L}]$.
Further, it is shown that the spectrum $\spec(L^{\rho_{m+L}}_0)$ of the 
conformal Hamiltonian $L^{\rho_{m+L}}_0$ in the sector $[\rho_{m+L}]$ equals
$\{\frac12\langle x,x\rangle : x\in m+L\}$.
Using the spin-statistics theorem \cite{GuLo1996} it follows:
\begin{prop}
  \label{prop:LatticeConformalNetsUMTC}
  Let $L$ be an even lattice, then $\Rep(\A_L)$ is braided equivalent to 
  $\cC(G_L,q_L)$.
\end{prop} 
Thus the classical result \cite{Ni1979} about lifting metric groups to even 
lattices, implies the following reconstruction result for conformal nets,
which expect to be well-known to experts.
\begin{thm}
  \label{thm:AllPointed}
  Let $\cC$ be a pointed UMTC, then there is an even lattice $L$, such that 
  $\Rep(\A_L)$ is braided equivalent to $\cC$.
\end{thm}
\begin{rmk} The same is true for vertex operator algebras using that for $V_L$ 
  the vertex operator algebra associated with the lattice $L$ the category of 
  $V_L$ modules is braided equivalent to $\cC(G_L,q_L)$ by \cite{DoLe1994}, 
  see also \cite{Hh2000}.
\end{rmk}

\subsection{Realization of Tambara--Yamagami doubles for odd groups}
In this section we will prove the following main reconstruction theorem.
\begin{thm} 
  \label{thm:AllOddTYs}
  Let $\cF$ be a unitary fusion category with Tambara--Yamagami fusion rules of
  even rank. 
  Then:
  \begin{enumerate}
    \item There is an even lattice $L$, and a proper $\ZZ_2$-action on $\A_L$, 
      such that the category of $\ZZ_2$-twisted representations 
      $\tRep{\ZZ_2}(\A_L)$ is tensor equivalent to $\cF$.
    \item\label{it:Double} There is an even lattice $M=L\bar L$ of rank 
      $0\pmod 8$, and a proper $\ZZ_2$-action on $\A_M$, such that 
      $\Rep(\A_M^{\ZZ_2})$ is braided equivalent to $Z(\cF)$.
    \item There is a self-dual even lattice $\Gamma$ of rank $0\pmod 8$, and a 
      proper action of the hypergroup $K_\cF$ on $\A_\Gamma$, such that 
      $\Rep(\A_\Gamma^{K_\cF})$ is braided equivalent to $Z(\cF)$.
  \end{enumerate}
\end{thm}

Associated with an abelian group $G$ and a non-degenerate symmetric bicharacter
$\bislotslot\colon G\times G \to \TT$ there are two unitary fusion category 
$\TY(G,\bislotslot,\pm)$ with irreducible sectors 
$\{[\rho_g],[\rho_\pm]:g\in G\}\cong G\cup\{\rho_\pm\}$ having the following 
fusion rules \cite{Iz2001II}: 
\begin{align}
  [\rho_g][\rho_h]&=[\rho_{g+h}]\,,&
  [\rho_\pm][\rho_g]&=[\rho_g][\rho_\pm]=[\rho_\pm]\,, &
  [\rho_\pm][\rho_\pm] &= \bigoplus_{g\in G}[\rho_g]\,.
\end{align}
Every unitary fusion category with this 
fusion rules is of the above form by the classification \cite{TaYa1998}.

Let us consider the Drinfel'd center $Z(\TY(G,\bislotslot,\pm))$. 
The objects are \cite{Iz2001II}:
\begin{gather}
  \label{eq:TYobjs}
  \begin{aligned}
  [\rho_g^i]&=(\rho_g,\hb_g^i)\,, & g\in G,\ i=0,1\,,\\
  [\rho_\pm^{(g,i)}]&=(\rho_\pm,\hb_\pm^{(g,i)})\,, & g\in G,~i=0,1\,,\\
  [\sigma_{g,h}]&=(\rho_g\oplus\rho_h,\hb_{g,h})\,, &g<h,~g,h \in G\,.
  \end{aligned}
\end{gather}
The modular data is given by \cite{Iz2001II}:
\begin{subequations}
\label{eq:TYMD}
\begin{align}
  \label{eq:TYS}
  S&=\frac1{2n}\quad
  \bordermatrix{
    &\scriptstyle[\rho^j_h] & \scriptstyle[\rho_\pm^{(h,j)}] &\scriptstyle[\sigma_{h,k}]\cr
    \scriptstyle[\rho^i_g] &\overline{\bi gh}^2&\sqrt n (-1)^i\overline{\bi gh}&2 \overline{\bi g{h+k}}\cr
    \scriptstyle[\rho_\pm^{(g,i)}]&\ast &\!\!\!\!(-1)^{i+j}\omega_g\omega_h\sum_k\bi{k-(g+h)}k\!\!\!\!&0\cr
    \scriptstyle[\sigma_{h',k'}]&\ast&\ast &2(\bi k{h'}\bi h{k'}+\bi k{k'}\bi h{h'})
  }\,,
  \\
  \label{eq:TYT}
  T&=\diag 
  \bordermatrix{
    & \scriptstyle[\rho^i_g] & \scriptstyle[\rho_\pm^{(g,i)}] &\scriptstyle[\sigma_{h,k}]\cr
    \!\!\!\! &\bi gg & (-1)^i\omega_g & \bi hk
  }
  \,.
\end{align}
\end{subequations}
Here $\omega_g$ is defined as follows. 
Let $a\colon G\to \TT$ be a function satisfying $a(g)a(h)=\bi gh a(g+h)$ and 
$a(g)=a(-g)$ for all $g,h\in G$. 
Let $\hat a$ be the finite Fourier transform of $a$ given by:
\begin{align} 
  \hat a(g)&=\frac{1}{\sqrt{|G|}}\sum_{h\in G}\overline{\bi gh} f(h)
\end{align}
and define $\omega_g=\sqrt{\pm\hat a(g)}$.

The subcategory $\cG:=\langle \rho_g^0:g\in G\rangle$ is a pointed subcategory.
If $|G|$ is of odd order $\cG$ is a UMTC, namely it is braided equivalent to 
$\cC(G,\bar q)$, where $\bar q(a)=\bi aa$.
\begin{prop}
  \label{prop:TYCentral}
  Let $G$ be an abelian group of odd order, then $Z(\TY(G,\bislotslot,\pm))$ is
  braided equivalent to $\MP(G,\bislotslot,\pm)\boxtimes \cC(G,\bar q)$. 
  Here $\bar q(a)=\bi aa$ and $\MP(G,\bislotslot,\pm)$ is a 
  unitary modular tensor category 
  of rank $4+(|G|-1)/2$ with global dimension $4|G|$.
  
  The inclusion functor $\iota\colon \cC(G,\bar q) \to \TY(G,\bislotslot,\pm)$ 
  mapping $[g]$ to $[\rho_g^0]$ is central.
  
  If there is a central injective functor 
  $\cC(G, \bar q) \to \TY(G',\bislotslot',\pm)$ with $|G'|=|G|$, then  
  $(G,\bar q)\sim (G', \bar q')$.
\end{prop}
\begin{proof}
  That the functor $\iota$ is central can be seen from the half-braidings. 
  We define $\MP(G,\bislotslot,\pm)$ to be the Müger centralizer 
  $\iota(\cC(G,\bar q))'\cap Z(\cF)$ which is modular by 
  \cite[Theorem 4.2]{Mg2003-MC}.

  Conversely, assume we have an injective central functor 
  $\iota\colon \cC(G, \bar q) \to \TY(G',\bislotslot',\pm)$. 
  From the modular data we see that the pointed part of 
  $\TY(G',\bislotslot',\pm)$ has $\ZZ_2\times G'$ fusion rules
  and because $|G'|$ is odd, the repletion of 
  $\iota(\cC(G',\bar q'))\subset Z( \TY(G',\bislotslot',\pm) )$ is the unique 
  pointed fusion subcategory of rank $|G'|$. 
  This gives a braided injective functor $\cC(G,q)\to \cC(G',q')$, 
  Since $|G'|=|G|$ it is a braided equivalence and we can conclude 
  $(G,q)\sim(G',q')$.
\end{proof}
\begin{prop} 
  \label{prop:PointedZ2Orbifold}
  Let $\A$ be a completely rational conformal net with a proper action of 
  $\ZZ_2$.
  If $\Rep(\A)$ is braided equivalent to $\cC(G,q)$ for some odd abelian group 
  $G$ and $\A$ has a $\ZZ_2$-twisted solition of dimension $\sqrt{|G|}$, then 
  the category of $\ZZ_2$-twisted solitons $\tRep{\ZZ_2}(\A)$
  is equivalent to $\TY(G,\bislotslot_q,\nu)$, where $\nu\in\{\pm\}$ and 
  $\bislotslot_q\colon G \to \TT$ is the unique non-degenerate bicharacter
  determined by $\bar q(g)=\bi g g_q$ for all $g\in G$.

  Furthermore, there is a proper $\ZZ_2$-action on 
  $\tilde \A=\A\otimes\A_{E_8}$, such that the category $\tRep{\ZZ_2}(\tilde\A)$
  is equivalent to $\TY(G,\bislotslot_q,-\nu)$.
\end{prop}
\begin{proof}
  We have $\Rep(\A)\subset \cF$ and $\Dim(\cF)=2|G|$, thus 
  $\Irr(\cF)\cong G\cup\{\rho\}$ and since $\cF$ is 
  $\ZZ_2$-graded we have $[g][\rho]=[\rho]=[\rho][g]$ and therefore by 
  Frobenius reciprocity $[\rho][\rho]=[\rho][\bar \rho]=\bigoplus_{g\in G} [g]$,
  thus $\cF$ has Tambara--Yamagami fusion rules.
  
  Using \cite[Proposition 5.11]{Bi2016} we have a central injective functor 
  $\rev{\Rep(\A)}\cong \cC(G,\bar q)\hookrightarrow \cF$. 
  Thus $\cF$ is equivalent to $\TY(G,\bislotslot_q,\pm)$ by Proposition 
  \ref{prop:TYCentral}. 

  The last statement follows from Proposition \ref{prop:ChangeFS} by 
  considering the $\ZZ_2$-orbifold $\widehat{\A^{\ZZ_2}} 
    = (\A\otimes\A_{A_1\lattimes E_7})\rtimes \ZZ_2 \subset \tilde \A 
    = \A\otimes\A_{E_8}$ 
  using the same argument as in Section \ref{sec:ChangeH3}.
\end{proof}
\begin{rmk}
  One can also check, \eg by inspecting the modular data or using 
  $\alpha$-induction,
  that the Frobenius--Schur indicators of generating object of the 
  representation category of the $\ZZ_2$-orbifold net coincides with the 
  Frobenius--Schur indicators of the Tambara--Yamagami category. 
  This is an alternative proof of the fact that $\A$ and $\tilde \A$ give 
  Tambara--Yamagami categories with opposite Frobenius--Schur indicator. 
\end{rmk}
By Theorem \ref{thm:AllPointed} there is an even positive lattice $L$, such
that $\A_L$ realizes $\cC(G,q)$, \ie $\Rep(\A_L)$ is braided equivalent to 
$\cC(G,q)$. 
\begin{prop} 
  \label{prop:Reflection2}
  Let $(G,q)$ be a metric group with $|G|$ odd and let $L$ be an even lattice,
  such that $\Rep(\A_L)$ is braided equivalent to $\cC(G,q)$.
  Let $\sigma$ be the reflection of $L$ and 
  $\A_L^{\ZZ_2}:=\A_L^{\langle\sigma\rangle}$ the associated $\ZZ_2$-orbifold.
  Then there is an irreducible $\ZZ_2$-twisted soliton $\rho$ of $\A_L$,
  with dimension $\sqrt{n}$.
  In particular, the category of $\ZZ_2$-twisted solitons of $\A_L$  
  is equivalent to the Tambara--Yamagami category $\TY(G,\bislotslot,\pm)$ 
  with bicharacter $\bislotslot$ determined by $q(g)=\bi gg^{-1}$.
\end{prop}
\begin{proof}
  By Proposition \ref{prop:PointedZ2Orbifold} we only have to find an 
  irreducible $\ZZ_2$-twisted solition $\rho$ with $d\rho=\sqrt{|G|}$.

  Indeed, following the notation in \cite{DoXu2006} 
  consider $(L^\ast/L)^\sigma$ to be the cosets fixed under $\sigma$.
  We get $(L^\ast/L)^\sigma= \{ x\in L^\ast : 2x\in L\}/L=\{L\}$
  and $|L^\ast/L |/|(L/L^\ast)^\sigma|=|G|$.
  Then it follows from \cite[Proposition 4.25, Corollary 4.31]{DoXu2006}, 
  that there is an irreducible soliton $\beta$ with $d\beta=\sqrt {|G|}$.
\end{proof}
Now we are able prove Theorem \ref{thm:AllOddTYs}. 
\begin{proof}[Proof of Theorem \ref{thm:AllOddTYs}]
  Let $G$ be an odd abelian group and $\bislotslot$ a non-degenerate
  symmetric bicharacter, which is determined by its quadratic form 
  $\bar q(g)=\bi gg$.
  By Theorem \ref{thm:AllPointed} we can choose even lattices $L_+$,  
  $L_- = L_+ \lattimes E_8$ and $\bar L$, such that 
  $\Rep(\A_{L_\pm})$ realizes $\cC(G,q)$ and $\Rep(\A_{\bar L})$
  realizes $\cC(G,\bar q)$. 
  Using the $\ZZ_2$-action on $\A_{L_+}$ as in 
    Proposition \ref{prop:Reflection2}  
  and the associated $\ZZ_2$-action on $\A_{L_-}\cong A_{L_+}\otimes \A_{E_8}$     as in Proposition \ref{prop:PointedZ2Orbifold} we get the first statement.
  
  The other statements follow from general theory. 
  For example, we have $\Rep(\A_{L_\pm}^{\ZZ_2}\otimes\A_{\bar L})$ is braided 
  equivalent to $Z(\TY(G,\bislotslot,\pm\nu))$, respectively, and we can choose
  $M_\pm = L_\pm \lattimes \bar L$.
  For the third statement the self-dual lattices are obtained by diagonally 
  gluing $\Gamma=(L_\pm\lattimes \bar L)\oplus \Delta(G)$, respectively.
\end{proof}

\section{Realization and classification of 
  generalized metaplectic modular categories}
\subsection{Generalized metaplectic modular categories from 
  Tambara--Yamagami categories}
\label{ssec:ModularDataMP}
Let $G=(G,+)$ be an abelian finite group of odd order with a 
symmetric non-degenerate
bicharacter $\bislotslot$ and $q(g)=\bi gg ^{-1}$.
It is convenient to choose a set $G_+$ of ``positive elements'' of $G$, 
such that  $G=G_+\sqcup \{0\} \sqcup -G_+$. 
We denote by $|g|$ the element $\pm g\in G_+\sqcup \{0\}$.

As in Proposition \ref{prop:TYCentral} we denote by $\MP(G,\bislotslot,\pm)$
the unitary modular tensor category given by the Müger centralizer 
$\cC(G,\bar q)'\cap Z(\TY(G,\bislotslot,\pm))$.
From the braided equivalence:
\begin{align} 
  Z(\TY(G,\bislotslot\slot,\pm) & \cong\MP(G,\bislotslot,\pm)\boxtimes \cC(G,\bar q)\,.
\end{align}
we can read off that 
$\MP(G,\bislotslot,\pm)$ has rank $(|G|-1)/244$
and that
$\Irr(\MP(G,\bislotslot,\pm)) 
  =\{ \rho_0^i,\rho^{(0,i)}_\pm,\sigma_{h,-h} : i=0,1 ; h\in G_+\}$.
From the modular data \eqref{eq:TYMD}
of $Z(\TY(G,\bislotslot,\pm))$ we get
the following modular data of $\MP(G,\bislotslot,\pm)$:
\begin{subequations}
\begin{align}
  S&=\frac1{2\sqrt n}\quad
  \bordermatrix{
    &\scriptstyle[\rho^j_0] & \scriptstyle[\rho_\pm^{(0,j)}] &\scriptstyle[\sigma_{h,-h}]\cr
    \scriptstyle[\rho^i_0] &1&\sqrt n (-1)^i &2 \cr
    \scriptstyle[\rho_\pm^{(0,i)}]&\ast &\!\!\!\!(-1)^{i+j}\omega_0\omega_0\sum_k\bi kk\!\!\!\!&0\cr
    \scriptstyle[\sigma_{k,-k}]&\ast&\ast &2\left(\overline{\bi h k}^2+{\bi hk}^2\right)
  }\,,\\
  T&= \e^{-\pi\ima c_{\mathrm{top}}/12} \cdot \diag 
  \bordermatrix{
    & \scriptstyle[\rho^i_0] & \scriptstyle[\rho_\pm^{(0,i)}] &\scriptstyle[\sigma_{h,-h}]\cr
    \!\!\!\! &1  & (-1)^i\omega_g
     & \overline{\bi hh}
  }\,,
\end{align}
\end{subequations}
where $c_{\mathrm{top}} \pmod 8$ is determined from $(G,q)$ by 
\eqref{eq:GaussSum}.
Thus the fusion rules of $\MP(G,\bislotslot,\pm)$ denoting 
$[\id]:= [\rho_0^0]$ and $[\alpha]:= [\rho_0^1]$ are
\begin{gather}
  \begin{aligned}
    [\alpha ]^2 &= [\id]\,, &
    [\alpha][\rho_\pm^{(0,0)}]&= [\rho_\pm^{(0,1)}]\,,\\
    [\alpha][\rho_\pm^{(0,1)}]&= [\rho_\pm^{(0,0)}]\,,&
    [\alpha][\sigma_{h,-h}]&= [\sigma_{h,-h}]\,,\\ 
    [\rho_\pm^{(0,0)}][\rho_\pm^{(0,0)}]&=[\id]+\sum_{-h<h} [\sigma_{h,-h}]\,,&
    [\rho_\pm^{(0,0)}][\rho_\pm^{(0,1)}]&=[\alpha]+\sum_{-h<h} [\sigma_{h,-h}]\,,\\
    [\rho_\pm^{(0,0)}][\sigma_{h,-h}]&=[\rho_\pm^{(0,0)}]+[\rho_\pm^{(0,1)}]\,,&
    [\rho_\pm^{(0,1)}][\rho_\pm^{(0,1)}]&=[\id]+\sum_{-h<h} [\sigma_{h,-h}]\,,\\
    [\rho_\pm^{(0,1)}][\sigma_{h,-h}]&=[\rho_\pm^{(0,0)}]+[\rho_\pm^{(0,1)}]\,,&
    [\sigma_{h,-h}][\sigma_{h,-h}]&=[\id]+[\alpha]+[\sigma_{2h,-2h}]\,,\\ 
    [\sigma_{g,-g}][\sigma_{h,-h}]&=[\sigma_{g+h,-g-h}]+ [\sigma_{g-h,h-g}]\,.
\end{aligned}
\end{gather}

\begin{example}
  \label{ex:Spin}
  For $p\in \NN$ consider the unitary modular tensor category
  $\umtc{\Spin(2p+1)_2}=\cC(B_{p,2})$ 
  (often called $\umtc{\SO(2p+1)_2}$ in the literature).
  Let $\rho\in\umtc{\Spin(2p+1)_2}$ be one of the two irreducible generating 
  objects with dimension $\sqrt{2p+1}$, then one can calculate from the 
  modular data that $\nu_\rho=(-1)^{\lfloor\frac{p+1}2\rfloor}$.
  Since $(\umtc{\Spin(2p+1)_2})_{\ZZ_2}^0$ is braided equivalent to 
  $\umtc{\SU(2p+1)_1}$ it follows that we have a braided equivalence
  \begin{align}
    \label{eq:Spin}
     \umtc{\Spin(2p+1)_2} &\cong
     \MP\left(\ZZ_{2p+1},\bislotslot_{A_{2p}},(-1)^{\lfloor\frac{p+1}2\rfloor}\right)
  \end{align}
  where 
  $\bi xy_{A_{2p}}=\exp(-\frac{\pi \ima p xy}{p+1})$.

  Namely, the modular data of this category is given by the 
  Kac--Peterson $S,T$-matrices 
  \cite{KaPe1984} with $c=2p$ and because 
  $\umtc{\Spin(2p+1)_2}$ is a metaplectic modular category
  which are classified by $S,T$, see \cite{ArChRoWa2016} and below. 
  This clarifies the relationship between doubles of Tambara--Yamagami
  categories and $\cC(B_{p,2})$ quantum group categories.
\end{example}

\begin{example} 
  Consider the $A_{2p}$ lattice. 
  Then $A_{2p}^\ast/A_{2p}\cong \ZZ_{2p+1} :=\{-p,\ldots,+p\}$ with bicharacter
  and $\bi xy_{A_{2p}}=\exp(-\pi \ima xy p/(+1))$ and quadratic form 
  $q(x)=\langle x,x\rangle_{A_{2p}}^{-1}$ and $\Rep(\A_{A_{2p}})$ is 
  braided equivalent to $\cC(\ZZ_{2p+1},q)$.
  Then $\A_{A_{2p}}\cong \A_{\SU(2p+1)_1}$ \cite{St1995,Xu2009,Bi2012} and 
  $\A_{A_{2p}}^{\ZZ_2}\cong \A_{\Spin(2p+1)_2}$ by \cite{Xu2000-2}.
\end{example}
As an application we get the following.
\begin{prop} 
  Let $p\in \NN$, then $\Rep(\A_{\Spin(2p+1)_2})$ is braided equivalent to 
  $\umtc{\Spin(2p+1)_2}$. 
  In particular, $\A_{\Spin(2p+1)_2}$ is modular 
  (see Definition \ref{defi:ModularNet}).
\end{prop}
\begin{proof}
  It follows directly that $\Rep(\A_{\Spin(2p+1)_2})$ is braided equivalent
  to $\MP(\ZZ_{2p+1},\bislotslot_{A_{2p}},\nu_\rho)$ and by \eqref{eq:Spin}
  we have to only check that the Frobenius--Schur indicator of the generating
  object equals $\nu_\rho=(-1)^{\lfloor\frac{p+1}2\rfloor}$.
  But the value of $\nu_\rho$ is distinguished by the twist $\omega_{\rho}$ 
  thus by the $T$-matrix.
  Equivalently, one can use Bantay's formula \eqref{eq:Bantay}. 
  Thus it is enough to check that the net is modular.
  
  The central charge of $\A_{2p}$ and thus of thus $\A_{\Spin(2p+1)_2}$
  is $c=2p$.
  Since every irreducible representation in $\Rep(\A_{\Spin(2p+1)_2})$ comes 
  from a module of the affine Lie algebra and 
  $T_{\rho\rho}=\diag( \e^{2\pi\ima (L^\rho_0-p/12)})$ by the 
  Guido--Longo spin statistics theorem \cite{GuLo1995} the $T$-matrix coincides
  with the Kac--Peterson $T$-matrix.
  Since also the $S$-matrices coincides, we get that that $\A$ is modular.
\end{proof}

\subsection{Classification of generalized metaplectic modular categories}
Let $n$ be odd and $G=\ZZ_n$, then the fusion rules of $\MP(G,\bislotslot,\pm)$
are the fusion rules of $\Spin(n)_2$, see Example \ref{ex:Spin}. 
Any unitary modular tensor category with these fusion rules %
is called a \myemph{metaplectic modular category} and 
metaplectic modular categories up to braided equivalence have been classified
in \cite{ArChRoWa2016}.
Here we allow $G$ to be an arbitrary abelian group of odd order and make the 
following definition.
\begin{defi} 
  \label{defi:GeneralizedMetaplecticModularCategory}
  Let $G=(G,+)$ be a finite abelian group with $|G|$ odd.
  A unitary modular tensor category of rank $(|G|+7)/2$ with irreducibles 
  $\{[1],[\alpha],[\rho],[\alpha\rho],[\sigma_g] : g\in G_+\}$ having the 
  commutative fusion rules
  \begin{gather}
    \label{eq:genMPFusionRules}
    \begin{aligned}
      [\alpha]^2&=[1]\,,
      &
      [\alpha][\sigma_g]&=[\sigma_g]\,,
      &
      [\rho]^2&=[1]+\sum_{g\in G_+}[\sigma_g]\,,
      \\
      [\alpha][\rho]&=[\alpha\rho]\,,
      &
      [\sigma_g][\sigma_{g}]&= [1]+[\alpha]+[\sigma_{|2g|}]\,, %
      &
      [\sigma_g][\sigma_h] &=[\sigma_{|g+h|}]+[\sigma_{|g-h|}]\,,
    \end{aligned}
  \end{gather}
  for all $g,h\in G_+$ with $g\neq h$ is called a 
  \myemph{generalized metaplectic modular category} (based on $G$).
\end{defi}
Thus by Subsection \ref{ssec:ModularDataMP} we get that 
a generalized metaplectic modular category based on $G$
is a unitary modular tensor category which has the same fusions rules as  
$\MP(G,\bislotslot,\pm)$.
In particular, $\MP(G,\bislotslot,\nu)$ 
is a generalized metaplectic modular category based on $G$.
\begin{lem}
  \label{lem:TrivialTwist}
  Let $\cC$ be a generalized metaplectic modular category, then $\alpha$ has 
  trivial twist $\omega_\alpha=1$. 
  In particular, $\cC$ contains a unique Tannakian subcategory
  $\Rep(\ZZ_2)$.
\end{lem}
This statement has been proven for $G=\ZZ_n$ in \cite{ArChRoWa2016}.
\begin{proof}
  From the fusion rules it follows that $\langle \alpha\rangle$ is 
  braided equivalent to $\cC(\ZZ_2,q)$ with $q(g)=\omega_\alpha^{g^2}$ 
  where $\omega_\alpha$ is the twist of $[\alpha]$ and $\omega_\alpha^4=1$.

  From $Y_{\alpha,\sigma}
    =\sum N_{\alpha,\sigma}^\rho \frac{\omega_\alpha\omega_\sigma}{\omega_\rho}
    d\rho$
  we get $Y_{\alpha,\alpha}=\omega_a^2$ and $Y_{\alpha,\sigma}=2\omega_\alpha$ 
  for every $\sigma=\sigma_h$.
  If we use that $Y_{\alpha,\bullet}$ are multiple of characters \cite{Re1989} 
  of the fusion rules, namely  $d_j^{-1}Y_{ij}Y_{kj}=\sum_m N_{ik}^m Y_{mj}$, 
  we get 
  $Y_{\sigma,\alpha}^2=Y_{\id,\alpha}+Y_{\alpha,\alpha}+Y_{\sigma',\alpha}$ 
  or equivalently 
  $4\omega_\alpha^2=1+\omega_\alpha^2+2\omega_\alpha$
  which has only the solution $\omega_\alpha=1$ fulfilling $\omega_\alpha^4=1$.
\end{proof}

\begin{lem}
  \label{lem:MPdeequivariantizationisTY}
  Let $\cC$ be a generalized metaplectic modular category for some odd abelian 
  group $G$, which has a unique Tannakian subcategory $\Rep(\ZZ_2)$ by 
  Lemma  \ref{lem:TrivialTwist}.
  Then $\cC_{\ZZ_2}$ is a Tambara--Yamagami category based on $G$ and 
  and $\cC_{\ZZ_2}^0\cong \cC(G, q)$ for some non-degenerate quadratic form 
  $q$ on $G$.
\end{lem}
\begin{proof}
  Let $\Theta$ be the (unique) Q-system with $\theta=[\id]\oplus [\alpha]$.
  Since $\alpha\otimes \sigma_h\cong\sigma_h$ we have that each free module 
  $\sigma_{h}\otimes \theta\in \cC_\Theta$ splits into two modules 
  $\beta_{ h}^\pm$.
  Let us write $G=\bigoplus_i \langle g_i \rangle$ with $g_i \in G_+$ of order 
  $n_i$. 
  Let $\cF_i\subset \cC_\Theta$ be the full subcategory 
  $\langle  \sigma_{g_i}\otimes\theta\rangle = \langle\beta_{g_i}^\pm\rangle$
  generated by $\sigma_{g_i}\otimes\theta$.
  It follows from the fusion rules that $\cF_i$ is a pointed category of rank 
  $1+2(n_i-1)/2=n_i$ and since the rank is odd and 
  $1 \prec \sigma_{g_i}\sigma_{g_i}$ it follows that  
  $\beta^+_{g_i}$ and  $\beta^-_{g}$ are each others inverses. 
  Thus $\cF_i=\langle \beta^+_{g_i}\rangle$ and $\cF=\bigoplus_i\cF_i$ has 
  $G$ fusion rules.

  The free $\Theta$-module $\tilde \rho=\rho\otimes \theta$ is irreducible and
  has dimension $\sqrt{|G|}$.
  From the fusion rules follows that
  $\tilde\rho\otimes_\Theta\tilde\rho\cong\bigoplus_{\alpha\in\Irr(\cF)} \alpha$.
  Since $\cC_\Theta^0=\cF$ is modular it must be braided equivalent to 
  $\cC(G,q)$ for a non-degenerate quadratic form $q$. 
\end{proof}

\begin{prop}
  \label{prop:FSalpha}
  The generating objects $[\rho]=[\rho^{(0,0)}_\pm]$
  and $[\alpha\rho]=[\rho_\pm^{(0,1)}]$ in $\MP(G,\bislotslot,\pm)$ 
  have Frobenius-Schur indicator given by $\nu_{\rho}=\nu_{\alpha\rho}=\pm1$,
  respectively.
\end{prop}
\begin{proof} 
  The $\alpha$-induction of the object $\rho$ is the generating object in the 
  Tambara--Yamagami category whose Frobenius--Schur indicator gives the sign
  $\pm$ of $\TY(G,\bislotslot,\pm)$, therefore it is enough to show that 
  $\nu_{\rho}=\nu_{\alpha^+_\rho}$.
  But if $\bar\rho=\rho$ and $\bar R_\rho=\nu_\rho R_\rho$ is a standard 
  solution, we have that $(R,\bar R)$ gives a standard solution for 
  $\alpha^+_\rho$ as for example in \cite[Lemma 2.2]{Re2000}.
\end{proof}

Let $(G,{\bislotslot})$ with $G$ a finite abelian group
and ${\bislotslot}$ a bicharacter on $G$.
We say $(G,{\bislotslot})$ and $(G',{\bislotslot}')$ are equivalent
if there is a group isomorphism $\phi\colon G \to G'$, such that
${\bi{\phi(g)}{\phi(h)}}'=\bi gh$ for all $g,h\in G$.
The same way, we say two bicharacters $\bislotslot$ and ${\bislotslot}'$ 
if there is a group automorphism $\phi$ of $G$,
such that
${\bi{\phi(g)}{\phi(h)}}'=\bi gh$ for all $g,h\in G$.

We have proven that every generalized metaplectic modular category $\cC$ is braided equivalent 
to $\MP(G,\bislotslot,\nu)$, where $\nu$ is the Frobenius--Schur indicator of
any irreducible generating object and $[(G,\bislotslot)]$ is fixed by 
either $Z(\cC_{\ZZ_2})$ being braided equivalent to 
$\cC\boxtimes \cC(G,\bar{q})$
or equivalently, 
by $\cC^0_{\ZZ_2}$ being braided equivalent to $\cC(G,q)$, 
where $\bar q(g)=\bi gg$. Therefore we have proven:
\begin{thm}
  \label{thm:MPclass1}
  Let $G$ be an abelian group of odd order. 
  Generalized metaplectic modular categories based on $G$ up to 
  braided equivalence are in one-to-one correspondence
  with pairs $([\bislotslot],\nu)$, where
  $[\bislotslot]$ is the equivalence class of a non-degenerate symmetric 
  bicharacter on $G$ and $\nu\in\{\pm\}$. 

  The correspondence is given by associating $\MP(G,\bislotslot,\pm)$ 
  with $([\bislotslot],\pm)$, respectively.
\end{thm}
\begin{rmk}
  Let $\cC$ be a generalized metaplectic modular category based on an 
  abelian group $G$ of odd order. 
  Then $[\bi\slot\slot]$ can be recovered from  $q(g)=\bi gg^{-1}$
  for $q$ a quadratic form on $G$, such that $\cC_{\ZZ_2}^0$ 
  is braided equivalent to $\cC(G,q)$. 
  The sign is given by the Frobenius--Schur indicator of the generating object 
  as in Proposition \ref{prop:FSalpha}.
\end{rmk}
Therefore the proof of Theorem \ref{thm:AllOddTYs} shows:
\begin{thm} 
  \label{thm:MPRealization}
  All generalized metaplectic modular categories (see Definition 
    \ref{defi:GeneralizedMetaplecticModularCategory})
  are realized by $\ZZ_2$-orbifolds of conformal nets associated with lattices.
\end{thm}

We have a decomposition into Sylow groups 
\begin{align}
  (G,\bislotslot) =\bigoplus_p (A_p,{\bislotslot}_p) 
   \quad\Longleftrightarrow \quad
  (G,q) =\bigoplus_p (A_p,q_p)
\end{align}
over primes $p>2$.
Each $(A_p,q_p)$ 
can be written as a direct sum of $(\ZZ_{p^k},q_{p^k,\pm})$ following 
\cite[Theorem 2.1]{BaJo2015}, \cf \cite{Wa1963,Ni1979}
where the two quadratic forms are given by $q_{p^k,\pm}(x) = \exp(2\pi \ima a x^2/p^k)$ for some $a$ with Jacobi symbol 
$( \frac ap )=\pm 1$, respectively.
Since with $k\geq 1$ the only relation \cite[Proposition 1.8.1 and 1.8.2]{Ni1979} is  
\begin{align}
  (\ZZ_{p^k},q_{p^k,\theta})\oplus
  (\ZZ_{p^k},q_{p^k,\theta})
  \cong (\ZZ_{p^k},q_{p^k,\theta'})\oplus
  (\ZZ_{p^k},q_{p^k,\theta'})
\end{align}
it follows that on  $\ZZ_{p^k}^n$ for $n\geq 1$ there are two isomorphism classes of metric groups:
$(\ZZ_{p^k}^n,q_{p^k,+})^{\oplus n}$ and $(\ZZ_{p^k}^n,q_{p^k,+})^{\oplus (n-1)}\oplus(\ZZ_{p^k}^n,q_{p^k,-})$.
Therefore, for fixed $G$ with 
\begin{align}
  G\cong \bigoplus_{i=1}^k \ZZ^{n_i}_{q_i} 
  \label{eq:oddAbelianGroup}
\end{align}
where $q_1,\ldots,q_k$ are distinct prime powers and 
$(n_1,\ldots,n_k)\in \NN^k$, 
there are exactly $2^k$ isomorphism classes of metric groups or equivalently 
bicharacters %
based on $G$.
Together, we have:
\begin{thm}
  \label{thm:MPclass2}
  Let $G$ be a finite abelian group of odd order with the decomposition into 
  $k$ summands as in (\ref{eq:oddAbelianGroup}), then there are $2^{k+1}$ 
  braided equivalence classes of generalized metaplectic modular categories 
  based on $G$.
\end{thm}
We remember that the modular data $(S,T)$ up to equivalence is an invariant
of a unitary modular tensor categery $\cC$.
It is has been believed that the modular data is a complete invariant %
but while writing this paper a counterexample has appeared in a preprint 
\cite{MiSc2017pp}.
Nevertheless, for generalized metaplectic modular categories we have the 
following:
\begin{thm}
  \label{thm:MPST}
  The modular data up to equivalence is a complete invariant for 
  generalized metaplectic modular categories, \ie 
  the braided isomorphism class of an odd generalized metaplectic modular 
  category is determined by its modular data $(S,T)$.
\end{thm}
\begin{proof} 
  Let $\cC$ be an odd generalized metaplectic modular category with modular 
  data $(S,T)$. The fusion rules of $\cC$ are determined by 
  $S$ via the Verlinde formula  \eqref{eq:Verlinde}.
  Let $G$ be the finite abelian group given by the fusion rules
  of objects of dimension as in the proof of 
  Lemma \ref{lem:MPdeequivariantizationisTY}.
  Then objects of dimension two are naturally indexed by $G_+$ which specifies 
  a unique non-degenerate bicharacter $\bislotslot$ such that 
  $\bi {\pm g}{\pm g} =T_{g,g}/T_{0,0}$ for all $g\in G_+$.
  Let $k$ be an object with dimension $\sqrt{|G|}$. 
  
  Then Bantay's formula \eqref{eq:Bantay} gives the Frobenius--Schur indicator 
  $\nu$ for this object and from the 
  classification it follows that $\cC$ is braided equivalent 
  to $\MP(G,\bislotslot,\nu)$.
\end{proof}

\subsection{(Generalized) metaplectic modular categories from condensation}
We remember that if $\Theta$ is a commutative Q-system in a 
unitary modular tensor category we get a new unitary modular tensor category 
$\cC_\Theta^0$. 
This process is also called \myemph{condensation}, since it corresponds to 
condensation in topological phases of matter. 
It also correspond to local extension by the Q-system $\Theta$ if 
$\cC$ is realized by a local conformal net.

We show that certain metaplectic modular categories can be obtained
from the basic examples $\umtc{\Spin(2n+1)_2}$ its reverse and semion 
categories using condensation by finite abelian groups, which are also
called \myemph{simple current extensions}.
These give simple relations in the Witt group of 
unitary modular tensor categories.
\begin{prop}
  Let $(G_i)_{i=1,\ldots,n}$ be a finite family of abelian groups of odd order
  and $\cC_i=\MP(G_i,\bislotslot_i,\nu_i)$.
  Then in 
  \begin{align}
    \cC&=%
    \mathop{\mathlarger{\mathlarger{\mathlarger{\boxtimes}}}}
    \limits_{i=1}^n \cC_i
  \end{align}
  we have a $\ZZ_2^n$ commutative Q-system. 
  Let $K\subset \ZZ_2^n$ be the subgroup of even codes.
  Then $\cC_K^0$ is braided equivalent to $\MP(G,\bislotslot,\nu)$, where
  \begin{align}
    G&=\bigoplus_{i=1}^n G_i\,,&\bislotslot&=  \bigoplus_{i=1}^n
    \bislotslot_i\,,
   &\nu&=\prod_{i=1}^n\nu_i\,.
  \end{align}
\end{prop}
\begin{proof}
  Taking a conformal net realization 
  $\A_L^{\ZZ_2^n}\cong\A^{\ZZ_2}_{L_1}\otimes\cdots\otimes \A^{\ZZ_2}_{L_n}$
  of $\cC$ from \ref{thm:MPRealization} let us consider the intermediate net 
  $\A_L^{\ZZ_2^n}\subset \A_L^{\Delta(\ZZ_2)}\subset\A_L\cong
  \A_{L_1}\otimes\cdots\otimes \A_{L_n}$ with 
  $\Delta(\ZZ_2)=\{(0,\ldots,0),(1,\ldots,1)\}\cong \ZZ_2$.
  On the one hand,  $\A_L^{\Delta(\ZZ_2)}$ is a simple current extension of 
  $\A_L$ by $K$, thus $\Rep(\A_L^{\Delta(\ZZ_2)})$ 
  is braided equivalent to $\cC_K^0$. 
  On the other hand,
  $\tRep{\Delta(\ZZ_2)}(\A_L)\cong \TY(G,\bislotslot,\nu)$ and therefore 
  $\Rep(\A_L^{\Delta(\ZZ_2)})$ is braided equivalent to
  $\MP(G,\bislotslot,\nu)$. 

  In purely categorical terms, we have that $\cC_{\ZZ_2^n}$ is equivalent to 
  \begin{align}
    \cF&=%
    \mathop{\mathlarger{\mathlarger{\mathlarger{\boxtimes}}}}
    \limits_{i=1}^n \TY(G_i,\bislotslot_i,\nu_i)
  \end{align}
  and there is an obvious injective functor
  $\TY(G,\bislotslot,\nu)\to \cF$. 
  Then similarly, one can check that there are braided equivalences
  $\cC_K^0 \cong  
    \TY(G,\bislotslot,\nu)^{\ZZ_2}\cong
    \MP(G,\bislotslot,\nu)$.
\end{proof}
\begin{example} 
  Let $p\neq 2$ be a prime number and $n\in\ZZ$.
  Then $\ZZ_{p^n}$ has up to equivalence two bicharacters.
  In, other words there are  up to equivalence two metric groups 
  $(\ZZ_{p^n},q_\pm)$.
  One is $q_+:=q_{A_{p^n-1}}$ from Table \ref{tab:lattices}.
  We have to distinguish two cases:
  \begin{enumerate}
    \item If $p\equiv 3 \pmod 4$, then we denote $q_-= \overline q_+$ which is 
      inequivalent to $q_+$. 
    \item If $p\equiv 1\pmod 4$, then $\overline q_+$ is equivalent to  $q_+$ 
      and therefore there exists an inequivalent quadratic form which we 
      denote by $q_-$.
  \end{enumerate}
  Namely, the quadratic form $q_+$ of $\ZZ_{p^n}$ can be represented as 
  $q_+(x)=\exp(2\pi \ima x^2 m /p^{n})$, where  $m$ is given by $p^n=2m+1$.
  Therefore $q_+\sim \overline q_+$ if and only if $m x^2 \equiv -m \pmod {p^n}$ 
  has a solution $x\not\equiv 0\pmod p$.
  But $x^2\equiv -1 \pmod{p^n}$ has a solution if and only if 
  $x^2\equiv -1 \pmod p$ has a solution due to Gauß. 
  By the \emph{Law of Quadratic Reciprocity}, $x^2\equiv -1 \pmod p$ has a 
  solution if and only if $p\equiv 1 \pmod 4$.
  So we conclude that $q_+\sim \overline q_+$ if and only if $p\equiv 1\pmod 4$.
\end{example}

This shows that many (generalized) metaplectic modular categories arise from 
condensations of $\umtc{\Spin(p^n)_2}^{\pm}$ and semion categories.
\begin{prop}
  We have:
  \begin{enumerate}
    \item 
      Let $N=\prod_{i=1}^r p_i^{n_i}$  with $p_i \equiv  3 \pmod 4$ for all 
      $i=1,\ldots, r$.
      Then all $2^{r+1}$ metaplectic modular categories for $\ZZ_N$ can be 
      obtained from condensing products of $\umtc{\Spin(p_i^{n_i})_2}^\pm$ and 
      $\cS^\pm$.
      Otherwise, at least $2^{r-k+1}$ of the $2^{r+1}$ 
      metaplection modular categories 
      arsies this way, where $k  =| \{i : p_i\equiv 1\pmod 4\}|$.
    \item Let $G\cong \bigoplus_{i=1}^r \ZZ^{n_i}_{q_i}$ be an odd abelian group with 
      $q_1=p_1^{n_1},\ldots,q_r=p_r^{n_r}$ distinct prime powers and $(n_1,\ldots,n_r)\in\NN^r$
      and let $k  =| \{i : p_i\equiv 1\pmod 4\}|$.
      Then at least $2^{r-k+1}$ of the $2^{r+1}$ generalized metaplection modular categories 
      arise from condensation of $\umtc{\Spin(p_i^{n_i})_2}^\pm$ and 
      $\cS^\pm$.
    \item
      All odd generalized metaplectic modular categories can be obtained from condensing 
      products of the following list of unitary modular tensor categories:
      \begin{itemize}
        \item $\umtc{\Spin(p^n)_2}^\pm$ with odd $p$ prime and $n\in \NN$,
        \item metaplectic modular categories $\MP(\ZZ_{p^n},\tilde q,+)$ with 
          $p$ odd prime and $n\in\NN$, 
          such that $p= 1 \pmod 4$ 
          (here $\tilde q$ is a non-degenerate quadratic 
           form with $\tilde q\not\sim q_+=q_{A_{p^n-1}}$), and
        \item the two semion categories $\cS^\pm$.
      \end{itemize}
  \end{enumerate}
\end{prop}
\begin{example}
  The 8 braided equivalence classes of (generalized) 
  metaplectic modular categories with $|G|=15$ (rank 11) are given by
  $\umtc{\Spin(15)_2}^\pm$, 
  the condensation 
    $(\umtc{\Spin(3)_2}^\pm\boxtimes\umtc{\Spin(5)_2})_{\ZZ_2}^0$
  and the four twists 
  $\hat\cC=(\cC\boxtimes\cS^+\boxtimes\cS^-)_{\ZZ_2}^0$ of these.
\end{example}
Everything in this subsection can be proved using only tensor categories.
We could have therefore proved the reconstruction results Theorem \ref{thm:MPRealization} and similarly Theorem \ref{thm:AllOddTYs} 
by proving the reconstruction only for cyclic groups of prime power orders.

\section{Several relations to generalized dihedral groups}
\label{sec:Dihedral}
\subsection{Generalized Tambara--Yamagami categories and generalized dihedral 
groups}
\label{ssec:GenTY}
Let $G=(G,\cdot)$ be a abelian group of odd order seen as a multiplicative
group.
Let us consider the \myemph{generalized dihedral group} 
$\Dih(G)=G\rtimes_{\alpha}\ZZ_2$, where 
$\Aut(G)\ni\alpha\colon g\mapsto g^{-1}$, 
\ie $\Dih(G)=G\sqcup G\tau$, with $\tau^2=e$ and $\tau g \tau = g^{-1}$.
Consider the following fusion rules of $\Dih(G)\cup \{\rho_+,\rho_-\}$:
\begin{gather}
  \label{eq:genTY}
    \begin{aligned}
      [\rho_\pm]^2&=\sum_{g\in G}[g]\,,&
      [\rho_\pm][\rho_\mp]&=\sum_{g\in G}[g\tau]\,,\\
      [g][\rho_\pm]&=[\rho_\pm][g] =[\rho_\pm]\,,&
     [\tau][\rho_\pm]&=[\rho_\pm][\tau]=[\rho_\mp]\,.
  \end{aligned}
\end{gather}
These fusion rules can be seen as a generalization of Tambara--Yamagami fusion
rules.
\begin{prop} 
  Let $\A$ be a completely rational conformal net with 
  $\Rep(\A)$ a generalized metaplectic modular category based on a finite
  abelian group $G$ of odd order and $\cB=\A\rtimes\ZZ_2$ the $\ZZ_2$-simple
  current extension (\eg for example $\A=\A_L^{\ZZ_2}\subset\cB= \A_L$ from 
  Theorem \ref{thm:MPRealization}).
  Then the category %
  of $\cB$--$\cB$ sectors coming from $\A\subset \cB$
  has the fusion rules \eqref{eq:genTY}.
\end{prop}
\begin{proof}
  Let $n=|G|$. 
  The modular invariant \cite{BcEvKa1999,BcEvKa2000}
  for the inclusion $A=\A(I)\subset B=\cB(I)$
  for the unique $\ZZ_2$-simple current extension 
  $\A \subset \cB=\A\rtimes \ZZ_2$
  can calculated to be:
  \begin{align}
    Z=  |\chi_{\id}+\chi_\alpha|^2 + \sum_{h\in G\setminus\{0\}} |\chi_{\sigma_{h,-h}}|^2 =
    |\chi_0+\chi_r|^2 + \sum_{i=1}^{\frac{n-1}2} 2|\chi_i|^2\,.
  \end{align}
  We get  
  $\tr(Z)=n+1$ and $\tr(ZZ^\mathrm{t})=2n+2$, which gives
  \cite[Corollary 6.10]{BcEvKa1999} that $|\Irr(\bim B\cC B)|=2n+2$.
  Further we have $\Dim \cC_{\ZZ_2}^0 =\Dim (\bim[0] B \cC B)=n$,  
  $\Dim \cC_{\ZZ_2}=\Dim (\bim[\pm] B \cC B)=2n$ and
  $\Dim(\bim{\ZZ_2}\cC{\ZZ_2})=\Dim (\bim B \cC B)=4n$.

  We have that $\cD_\pm:=\bim[\pm] B \cC B$ are Tambara--Yamagami categories, 
  see Lemma \ref{lem:MPdeequivariantizationisTY}. 
  Let us denote $\Irr(\bim[\pm]B\cC B)=G\cup \{\rho_\pm\}$. 
  Since $A\subset B$ is a fixed point under an outer $\ZZ_2$-action the 
  canonical endomorphism $\gamma\in \cD:=\bim {\ZZ_2}\cC{\ZZ_2}$ is of the form
  $[\gamma]=[\id]\bigoplus [\tau]$ for some $\tau\in \cD$ with $[\tau]^2=[\id]$.
  By a simple counting argument we know that 
  $\Irr(\cD) = G\cup {\rho_\pm}\cup \{\alpha_k:k=1,\ldots, n\}$,
  with $\alpha_k$ automorphisms.
  Since $\tau\not\in \cD_\pm$ see \eg \cite[Lemma 5.8]{Bi2016}, we get that 
  $[\tau] \in \{[\alpha_k] :k=1,\ldots, h\}$.
  
  The $\ZZ_2$-grading on $\cD_\pm$ gives a $\ZZ_2\times\ZZ_2$-grading on $\cD$. 
  Since $\cD =\langle\cD_+,\cD_-\rangle = \langle \rho_+,\rho_-\rangle$ we have 
  $[\tau]\prec [\rho_+][\rho_-]$, which is equivalent to 
  $[\tau][\rho_\pm]\prec [\rho_\mp]$. 
  Therefore we get $[\tau][\rho_\pm]= [\rho_\mp]$.
  But this implies $[\rho_+][\rho_-]=\sum_{g\in G}[g\tau ]$.
  Finally, $[\tau][g][\tau]=[g^{-1}]$ follows from how $\ZZ_2$ acts on $\cD_0$.
\end{proof}
The purely categorical formulation of this proposition is:
\begin{prop}
  \label{prop:GenTY}
  Let $\cC$ be a generalized metaplectic modular category based on $G$ with 
  $|G|$ odd,
  then there is a unique Tannakian subcategory $\Rep(\ZZ_2)$
  and $\cD=\bim{\ZZ_2}\cC{\ZZ_2}$ has the fusion rules 
  \eqref{eq:genTY}.
\end{prop}
The following is a classical extension problem of finite groups with cocycles.
We consider the extension of groups
\begin{align}
1\to A\to \Dih(A)\to \Dih(A)/A\cong \ZZ_2 \to 1\,.
\end{align}
For $[\omega] \in H^3(\Dih(A),\TT)$, by restriction we get elements 
$[\omega\restriction A]\in H^3(A,\TT)$
and $[\omega\restriction \langle \tau\rangle]\in H^3(\ZZ_2,\TT)$ for 
every order two element $\tau\in \Dih(A)$.

\begin{lem}
  \label{lem:Cocycle}
  Let $A$ be an odd abelian group.  
  Then the restriction map 
  $H^3(\Dih(A),\TT)\to H^3(A,\TT)\oplus H^3(\ZZ_2,\TT)$ is injective.
  In particular, there are two classes $[\omega_\pm]$ in 
  \begin{align}
    \{[\omega] \in H^3(\Dih(A),\TT) &: [\omega\restriction A]\in B^3(A,\TT)\}\,,
  \end{align}
  which are distinguished by 
  $[\omega_\pm\restriction \langle \tau\rangle]=[\pm] \in H^3(\ZZ_2,\TT)
    =\{[\pm1]\}$.
\end{lem}
\begin{proof}
  $H^3(\Dih(A),\TT)$ can be calculated by the second page of the 
  Lyndon--Hochschild--Serre spectral sequence
  \begin{align}
    E_2^{p,q}&=H^p(\ZZ_2,H^q(A,\TT)) \Longrightarrow H^{p+q}(\Dih(A),\TT)  
  \end{align}
  and since $A$ is odd one can calculate that 
  \begin{align}
    H^3(\Dih(A),\TT) &\cong E_2^{3,0}\oplus E_2^{0,3}
    \cong H^3(A,\TT)^\tau \oplus \ZZ_2
  \end{align}
  and there are no differentials in the spectral sequence since they would 
  connect $2$-groups with $p$-groups for odd $p$.
  Restriction to $A$ is the projection onto the first component and 
  restriction to 
  any subgroup of order 2 injects onto the second component.
  Alternative, one can show that the cohomology group in 
  \cite[Definition 2.3]{IzKo2002} vanishes.
\end{proof}
\begin{prop} 
  Let $A$ be an odd a abelian group.
  There are up to equivalence two pointed $\ZZ_2$-extensions 
    of $\Vect_A$ 
  with fusion rules given by $\Dih(A)$. They are 
  $\cF_+=
    \Vect_{\Dih(A)}$ and $\cF_-=\Vect^{\omega_-}_{\Dih(A)}$ 
  which are characterized by the Frobenius--Schur indicator $\nu_{\tau}=\pm1$
  for every order two element $\tau\in\cF_\pm$.
\end{prop}
\begin{proof}
  By Lemma \ref{lem:Cocycle} the only extensions are 
  $\Vect^{\omega_\pm}_{\Dih(A)}$,
  and for each order two element $\tau \in \Vect^{\omega_\pm}_{\Dih(A)}$
  we have that the subcategory $\langle \tau\rangle$ is equivalent to 
  $\Vect^{\pm}_{\ZZ_2}$ and therefore $\nu_\tau=\pm1$.
\end{proof}
\begin{prop}
  \label{prop:TrivialDihCocycle}
  Let $\cC$ be a generalized metaplectic modular category and 
  $\cD=\bim{\ZZ_2}\cC{\ZZ_2}$ as in Proposition \ref{prop:GenTY}.
  Then the pointed subcategory $\cD^\times\subset \cD$ is equivalent to 
  $\Vect_{\Dih(A)}$.
\end{prop}
\begin{proof}
  We have that $\cD^\times$ is equivalent to $\Vect^{\omega}_{\Dih(A)}$ 
  for some $[\omega]\in H^3(A,\TT)$.
  Since $\cD^0=\cC_{\ZZ_2}^0$ is braided equivalent to $\cC(A,q)$
  for some quadratic form $q$ and thus tensor equivalent to $\Vect_A$ by 
    Lemma \ref{lem:TrivialCocycle}, we get 
  $[\omega\restriction G]=0\in\Hom(G,\TT)$. 
  Since $[\id]+[\tau]$ is a Q-system we get that 
  $[\omega\restriction \langle \tau\rangle] = 0\in H^3(\ZZ_2,\TT)$. 
  This implies $[\omega] = 0\in H^3(\Dih(G),\TT)$ by Lemma \ref{lem:Cocycle}.
\end{proof}
\begin{prop} 
  \label{prop:RepDihedral}
  Let $A$ be an abelian group of odd order. 
  Let $\cC$ be a generalized metaplectic modular category based on $A$, 
  \ie $\cC\cong\MP(A,\bislotslot,\pm)$. 
  Then the even part $\cC_0$ of $\cC=\cC_0\oplus\cC_1$ is tensor equivalent to 
  $\Rep(\Dih(\hat A))$. 
  As a braided fusion category $\cC_0$ is degenerate with M\"uger center 
  braided equivalent to $\Rep(\ZZ_2)$.
\end{prop}
\begin{proof}
  By Theorem \ref{thm:MPRealization} there is an even lattice $L$ with 
  $L^\ast/L\cong A$ and a $\ZZ_2$-action on $\A_L$, such that 
  $\Rep(\A_L^{\ZZ_2})$ is braided equivalent to $\MP(A,\bislotslot,\pm)$. 
  Since $\Rep(\A_L)$ is tensor equivalent to $\Vect_A$ and in particular the 
  obstruction in $H^3(A,\TT)$ vanishes by Lemma  \ref{lem:TrivialCocycle}, we can consider the crossed produdct extension 
  $\A_L^{\ZZ_2}(I)\subset \A_L(I)\subset \A_L(I)\rtimes A$. 
  From the branching rules it easily follows that 
  $[\theta]=[\id]\oplus[\alpha]\oplus \bigoplus_{a\in A_+} 2[\sigma_a]$, 
  thus the inclusion has depth two and thus it is a crossed product by 
  an outer action of a Kac algebra by Ocneanu's characterization, see 
  \cite{Iz1993,Sz1994,Lo1994,Da1996,Sa1997,Iz1998}, 
  and which is co-commutative by 
  \cite[Corollary 9.9]{IzKo2002},
  thus it is a group subfactor $N^G\subset N$ by 
  \cite{VaKa1974,Iz1991,BaSk1993}.
  It is also of the form, $M^{\ZZ_2}\subset M \rtimes A$ with 
  $\langle A ,\ZZ_2\rangle =\Dih(A)$ with trivial 3-cocycle by Proposition 
  \ref{prop:TrivialDihCocycle}.
  Thus we can conclude that $G=\Dih(\hat A)$.
\end{proof}

\subsection{Doubles of generalized dihedral groups}
\label{ssec:DoublesOfGenDihGroups}
In this section we want to clarify the relation between doubles of generalized 
dihedral groups and generalized metaplectic modular categories. 

Let $A$ be an abelian group of odd order. 
Using the Galois correspondence of Longo--Rehren subfactors \cite{Iz2000}, 
we obtain doubles of $\Dih(A)$.  
Therefore, let $\cC$ be a generalized metaplectic modular category based on $A$,\ie $\cC$ is braided equivalent to $\MP(A,\bislotslot,\nu)$/
There is a unique Tannakian subcategory $\Rep(\ZZ_2)\subset \cC$. 
We have braided equivalences 
$Z(\bim{\ZZ_2}\cC{\ZZ_2})\cong Z(\cC)\cong \cC\boxtimes\rev{\cC}$.
Further, $\Vect_{\Dih(A)}\subset \bim{\ZZ_2}\cC{\ZZ_2}$ by 
  Proposition \ref{prop:TrivialDihCocycle}.
Using the Galois correspondence, we have a Tannakian subcategory 
$\Rep(\ZZ_2)\subset \cC\boxtimes\rev{\cC}$, such that
$(\cC\boxtimes \rev{\cC})_{\ZZ_2}^0\cong Z(\Vect_{\Dih(A)})$.
We will show that it is possible to twist this construction and to obtain all 
twisted doubles of $\Dih(A)$. 

Let $A$ be an abelian group and $\hat A=\Hom(A,\TT)$ the Pontryagin dual. 
The canonical pairing $q_\mathrm{can}\colon \hat A \oplus A \to \TT$
given by $q_\mathrm{can}(\chi,a)=\chi(a)$ is a non-degenerate quadratic form on
$\hat A \oplus A$. 
The Drinfel'd center $Z(\Vect_A)$ is braided equivalent to 
$\cC(\hat A\oplus A,q_\mathrm{can})$.
Let $(G,q)$ be a metric group, we say a subgroup $L\leq G$ is 
\myemph{Lagrangian} if $|L|^2=|G|$ and  $q\restriction L \equiv 1$.
\begin{prop} 
  \label{prop:MetricDoubleGroups}
  Let $A$ be an abelian group of odd order. 
  The following metric groups are equivalent:
  \begin{enumerate}
    \item $(A\oplus A,q\oplus \bar q)$, where $q$ is a non-degenerate bicharacter on $A$.
    \item $(A\oplus \hat A,q_{\mathrm{can}})$ where $q_\mathrm{can}$ is the canonical pairing.
    \item $(G,q)$ a metric group admitting a Lagrangian subgroup $L\cong \hat A$ 
      and $\cC(G,q)_{A}\cong \Vect_A$.
  \end{enumerate}
\end{prop}
\begin{proof} 
  For each of the $(G,q)$ we have that $\cC(G,q)$ is braided equivalent to 
  $Z(\Vect_A)$.
  Namely, in (1), using  Lemma \ref{lem:TrivialCocycle} we have 
  $Z(\Vect_A)\cong Z(\cC(A,q))\cong \cC(A\oplus A,q\oplus \bar q)$. 
  In (3), we have $\cC(G,q)$ that is braided equivalent to 
  $Z(\cC(G,q)_{\hat A})$, since $\hat A$ gives rise to a Lagrangian algebra 
  \cite{DaMgNiOs2013}.
\end{proof}
\begin{prop} Let $A$ be an abelian group and $(G,q)$ be a metric group based on $A$ as in Proposition \ref{prop:MetricDoubleGroups}.
  Then $\MP(G,q,\pm)$ is braided equivalent to $Z(\Vect_{\Dih(A)}^{\omega_\pm})$
where $[\omega_+]=0\in H^3(\Dih(A),\TT)$ and $[\omega_-]$ is the unique order two element in $H^3(\Dih(A),\TT)$.
\end{prop}
\begin{proof}
  Let 
  $\cC_\pm=\MP(A,\bislotslot,\pm)$ and we can consider the modular tensor category 
  $\cC:=(\cC_\mu\boxtimes\rev{\cC}_\nu)_{\ZZ_2}^0$.
  It follows that $\cC$ is braided equivalent to $\MP(A\oplus A,\bislotslot\oplus\overline{\bislotslot},\mu\nu)$.
  Namely, $\cC_\nu\boxtimes\rev{\cC}_{\mu}\cong \cF^{\ZZ_2\times\ZZ_2}$ with 
  $\cF=\TY(A,\bislotslot,\nu)\boxtimes\TY(A,\overline{\bislotslot},\mu)$ and $\cF$ contains the subcategory 
  $\cF_0\cong\TY(A\oplus A,\bislotslot\oplus\overline{\bislotslot},\pm)$
  and we have $\cF_0^{\ZZ_2}\cong (\cC_\mu\boxtimes\rev{\cC}_\nu)_{\ZZ_2}$. 
  There is a Lagrangian subgroup $L$ in the pointed modular tensor category
  $\cD=\cC_{\ZZ_2}^0$ with $\cD_{\hat L}\cong \Vect_A$, which by restriction gives a 
  Lagrangian algebra $\Theta$ in $\cC:=\MP(A\oplus A,\bislotslot\oplus\overline{\bislotslot},\mu\nu)$. 
  In the case, $\mu=\nu$ we have discussed above that
  $\cC_\Theta\cong \Vect_{\Dih(A)}$. 
  In the case, $\mu=-\nu$, the fusion rules do not change and we have 
  $\cC_\Theta\cong \Vect_{\Dih(A)}^\omega$ for some class 
  $[\omega]\in H^3(A,\TT)$.
  But $\Vect^\omega_{\Dih(A)}$ is an extension of $\Vect_A$ and therefore 
  $\omega\in[\omega_\pm]$. 
  Finally, the sign is determined by the Frobenius--Schur indicator as in
  Proposition \ref{prop:FSalpha}.
\end{proof}
The following example shows that we can also twist by certain elements in 
$H^3(A,\TT)$.
\begin{example}
  Take $\cC(\ZZ_9,q)\cong \umtc{\SU(9)_1}\cong\Rep(\A_{A_8})$, 
  \ie $q(x)=e^{\frac{8\pi i x^2}{9}}$. 
  Then there is a unique Lagrangian subgroup $A\cong \ZZ_3$. 
  It corresponds to the conformal embedding $\A_{A_8}\subset \A_{E_8}$. 
  One can check that $\umtc{\Spin(9)_2}$ is braided equivalent to 
  $Z(\Vect_{\Dih(\ZZ_3)}^\omega)$ for a generator
  $[\omega]\in H^3(\Dih(\ZZ_3),\TT)$. 
  We note that $\Dih(\ZZ_3)$ is isomorphic to the symmetric group $S_3$. 

  It follows that, $\Rep(\A_{\Spin(9)_2})$ is braided equivalent to 
  $Z(\Vect^\omega_{S_3})$. 
  We get the other twist by using that for $L=A_8^2\oplus \ZZ_3$
  the net $\A_L$ realizes $\cC(\ZZ_9,q')$ with $q'(x)=\e^{4\pi \ima x^2/9}$.
  Then $\A_L^{\ZZ_2}$ realizes $Z(\Vect^{\omega'}_{S_3})$
  for some $\omega'$ with 
  $[\omega'\restriction \ZZ_3]=[\bar\omega\restriction\ZZ_3]$.
  The trivial cocycle is realized by $\A^{\ZZ_2}_{A_2E_6}$.
  The other three elements $H^3(S_3,\TT)$ with non-trivial 2-part can be 
  obtained by the twisting as in Proposition \ref{prop:ChangeFS}.
  This way we can realize all twisted doubles of $S_3$ by a 
  $\ZZ_2$-orbifold $\A^{\ZZ_2}_L$ of a conformal net associated with a lattice.

  We also see that the six generalized metaplectic modular categories 
  \begin{align}
    &\MP(\ZZ_3^2,\bislotslot_\mathrm{can},\pm)\,,& 
    \MP(\ZZ_9,\bislotslot_\pm,\pm)
  \end{align}
  are all braided equivalent to some $Z(\Vect_{S_3}^\omega)$ 
  with all six possible cohomology classes $[\omega]\in H^3(S_3,\TT)$ arising.
  In particular, this shows that the twisted doubles of $S_3$ have 
  two different fusion rules depending 
  if $[\omega\restriction\ZZ_3]$ is trivial or not.
\end{example}

A UFC $\cF$ is called \myemph{group-theoretical} if
there is a finite group $G$ and $[\omega]\in H^3(G,\TT)$, such that
$Z(\cF)$ is braided equivalent to $Z(\Vect_G^\omega)$,
or equivalently,  $\cF$ is (weakly monoidally) 
Morita equivalent to $\Vect_G^\omega$ \cite{Mg2003,Mg2003II}.
In \cite[Theorem 4.6]{GeNaNi2009} it is shown that $\TY(G,\bislotslot,\pm)$ 
is group theoretical if and only if $(G,\bislotslot)$ contains a Lagrangian, \ie 
$L\leq G$ with $|L|^2=|G|$ and $\bislotslot\restriction L \equiv 1$. 
For $G$ odd this is equivalent to $L$ being a Lagrangian subgroup of $(G,q)$.

Assume $(G,q)$ is a metric group with a Lagrangian subgroup $\hat A$. 
Then $\cC(G,q)$ has a Tannakian subcategory $\Rep(A)$. 
The de-equivariantization $\cC(G,q)_{A}$ is tensor equivalent to 
$\Vect_{A}^\omega$ for some $[\omega]\in H^3(A,\TT)$ 
  \cite[Proposition 4.58]{DrGeNiOs2010}
and $\cC(G,q)$ is braided equivalent to $Z(\Vect_{A}^\omega)$
  \cite[Proposition 4.8]{DaMgNiOs2013}.

\begin{prop} 
  \label{prop:TwistedDoubleRealization}
  Let $(G,q)$ be a metric group of odd order with Lagrangian subgroup $\hat A$
  and $\bislotslot$ the non-degenerate bicharacter on $G$, such that 
    $q(x)=\bi x x ^{-1}$.
  Then $\MP(G,\bislotslot,\pm)$ is braided equivalent to 
    $Z(\Vect_{\Dih(A)}^{\omega_\pm})$ 
  for some $[\omega_\pm]\in H^3(\Dih(A),\TT)$.

  In this case, $C(G,q)_{\hat A}$ is tensor equivalent to $\Vect_{A}^\omega$
  with $[\omega]=[\omega_\pm\restriction A]$.
  Let $\tau \in \Dih(\hat A)$ with $\langle \tau\rangle \cong \ZZ_2$ then
  $[\omega_\pm\restriction \langle\tau\rangle]\cong [\pm]\in H^3(\ZZ_2,\TT)$, 
  respectively, \ie every simple self-dual element 
  $\tau\in \Vect_{\Dih( A)}^{\omega_\pm}$ has Frobenius--Schur indicator
  $\nu_\tau=\pm 1$, respectively.
\end{prop}
\begin{proof}
  Let $\A_L^{\ZZ_2}$ be a realization of $\MP(G,\bislotslot,\pm)$ from Theorem 
  \ref{thm:MPRealization}. 
  We can consider the simple current extension $\A_L\rtimes \hat A$ of $\A_L$ by
  $\hat A$, in other words the Lagrangian subgroup $\hat A\leq G=L^\ast / L$ 
  gives a self-dual lattice $\Gamma=L\oplus \hat A$ and $\A_L\rtimes \hat A$ can  be identified with $\A_\Gamma$. 
  The category of $A$ twisted representations of $\A_\Gamma$ is equivalent to 
  $\Vect_{A}^\omega$ for some $[\omega]\in H^3(A,\TT)$ by the above discussion
  or \cite[3.6 Corollary]{Mg2010}, see \cite[Theorem 1.7]{Bi2016}.
  By considering the inclusion $\A_L^{\ZZ_2}\subset\A_L\subset \A_\Gamma$ it 
  follows from the the branching rules that $\A_L^{\ZZ_2}\subset \A_\Gamma$
  \begin{align}
    [\theta]&=[\id]\oplus[\alpha]\oplus\bigoplus_{a\in \hat A_+}2[\sigma_a]
  \end{align}
  and the inclusion has depth 2, thus by \cite[Corollary 1.2]{Bi2016} we have 
  that $\A_L^{\ZZ_2}$ is an orbifold $\A_{\Gamma}^{H}$ of $\A_\Gamma$ for some 
  group $H\triangleright A$ of order $2|A|$. 
  A similar argument as in the proof of Proposition \ref{prop:RepDihedral} 
  shows that $H=\Dih(A)$.
  The category of $H$-twisted representations of $\A_\Gamma$ is a
  $\ZZ_2$-extension $\Vect_{H}^{\tilde\omega}$ of $\Vect_{ A}^\omega$ with 
  $\tilde \omega\restriction A = \omega$.
  As in Proposition \ref{prop:FSalpha} follows that the $\nu_\tau=\pm1$.
\end{proof}

\begin{prop} 
  \label{prop:SomeDihAreMeptaplectic}
  Let $A$ be an abelian group of odd order, 
  $[\tilde\omega]\in H^3(\Dih(A),\TT)$,
  and $[\omega]\in H^3(A,\TT)$ given by $\omega=\tilde \omega\restriction A$.
  Further, assume that $Z(\Vect^{\omega}_A)$ is pointed, \ie
  braided equivalent to  $\cC(G_{A,\omega},q_{A,\omega})$
  for some metric group $(G_{A,\omega},q_{A,\omega})$.
  
  Then $Z(\Vect_{\Dih(A)}^{\tilde\omega})$ is braided eqivalent to 
  $\MP(G_{A,\omega},\bislotslot_{A,\omega},\nu_\tau)$, where 
  $\nu_\tau$ is the Frobenius-Schur indicator of any order two element 
  $\tau\in\Vect_{\Dih(A)}^\omega$.
\end{prop}
\begin{proof} 
  We have a Tannakian subcategory $\Rep(\Dih(A))$ of
  $\cC:=Z(\Vect^{\tilde\omega}_{\Dih(A)})$, such that 
  $\cC_{\Dih(A)}\cong \Vect^{\tilde\omega}_{\Dih(A)}$.
  From the Galois correspondence we get a Tannakian subcategory category 
  $\Rep(\ZZ_2)$ of $\cC$, such that $\cC_{\ZZ_2}^0\cong Z(\Vect_A^\omega)$ 
  which by assumption is pointed.
  Thus $\cC_{\ZZ_2}$ is $\ZZ_2$-crossed braided extension $\cF$ of 
  $Z(\Vect_A^\omega)$ which is the equivalent to the relative center 
  $Z_{\Vect_A^\omega}({\Vect_{\Dih(A)}^{\tilde\omega}})$ by 
  \cite[Theorem 3.5]{GeNaNi2009}. 
  Let $(\rho,\varepsilon_\rho)
    \in Z_{\Vect_A^\omega}({\Vect_{\Dih(A)}^{\tilde\omega}})$ 
  irreducible, where $\varepsilon_\rho
    =\{\varepsilon_\rho(\sigma)\colon 
    \rho\sigma\to \sigma \rho\}_{\sigma\in \Vect_A^\omega}$ 
  and assume $[\tau]\prec [\rho]$. 
  Then because of the half-braiding, we have $[\rho][g]=[g][\rho]$ 
  for all irreducible sectors $[g]$ in $\Vect_A$.
  Thus $[g][\tau ][g^{-1}] =[g^2][\tau]\prec [\rho]$ and because $|A|$ is odd, 
  we have $\bigoplus_{g\in A}[g][\tau]\prec [\rho]$ and $d\rho \geq |A|$.
  Via a counting argument involving the global dimensions we get
  $\Irr(\cF)=G\cup\{(\rho,\varepsilon_\rho)\}$ and $d\rho=|A|$.
  Thus can conclude that $\cF$ is tensor equivalent to 
  $\TY(G,\bislotslot,\pm)$ with $q_{A,\omega}(g)=\bi g g$.
\end{proof}
\begin{rmk}
  \label{rmk:H3ab}
  We note that $Z(\Vect^{\omega}_A)$ is pointed if and only if 
  $[\omega] \in H^3(A,\TT)_{\mathrm{ab}}=\ker(\psi^\ast)$ 
  \cite[Corollary 3.6]{NgMa2001}, see also \cite[Proposition 4.1]{Ng2003}.
  Here $\psi^\ast\colon H^3(G,\TT)\to \Hom(\Lambda^3
  G,\TT)$ is  given by
  \begin{align}
    \left[\psi^\ast([\omega])\right](x,y,z)&=
      \prod_{\omega\in\mathbb{S}_3}\omega(\sigma(x),\sigma(y),\sigma(z))
        ^{\sign(\sigma)}
  \end{align}
  For example, let $G=(\ZZ_n)^3$ we have the cocycle 
  $\omega(x,y,z)=\zeta_n^{x_1y_2z_3}$ and 
  $Z(\Vect_{(\ZZ_n)}^\omega)$ is not pointed \cite[Example 4.5]{Ng2003}.
\end{rmk}
\begin{lem} 
  \label{lem:H3ab}
  Let $A$ be an abelian group of odd order, $\tilde\omega\in H^3(\Dih(A),\TT)$, 
  and $\omega =\tilde\omega\restriction A$. 
  Then $[\omega]\in H^3(A,\TT)_\mathrm{ab}$.
\end{lem}
\begin{proof}
  Let $\varphi=\psi^\ast\omega\in \Hom(\Lambda^3(A),\TT)$. 
  For any $c\colon A\times A \times A \to \TT$ let us denote by $c^\tau$ the map
  $c^\tau\colon A\times A \times A \to \TT$ with 
  $c^\tau(g,h,k)=c(g^{-1},h^{-1},k^{-1})$.

  Let $\tau\in\Dih(A)$ with $\tau^2=1$ and let 
  \begin{align}
    \xi(g,h)&=\frac{\tilde\omega(\tau,g,h)\tilde\omega(g^{-1},h^{-1},\tau)}
    {\tilde\omega(g^{-1},\tau,h)}
  \end{align}
  which correspond to the associator $(g^{-1}h^{-1})\tau\to \tau(gh)$ in 
  $\Vect_{\Dih(A)}^{\tilde\omega}$.
  Considering the associator $\omega(g,h,k)$ of 
  $\tau((gh)k)\to \tau( g(hk))$ and $\omega^\tau(g,h,k)$ of
  $((g^{-1}h^{-1})k^{-1})\tau\to(g^{-1}(h^{-1}k^{-1}))\tau$ 
  in $\Vect_{\Dih(A)}^{\tilde \omega}$.
  We get 
  (similarly to \cite[Lemma 2.5]{Iz2016}) that $\omega =\omega^\tau  b_\xi$ with
  $b_\xi(g,h,k) =\frac{\xi(h,k)\xi(g,hk)}{\xi(gh,k)\xi(g,h)}$ for the above 
  $\xi\colon A\times A \to \TT$. 
  From the symmetry of $b_\xi$ follows $\varphi =\varphi^\tau$ and since 
  $\varphi \in \Hom(\Lambda^3(A),\TT)$, we have $\varphi^\tau=\varphi^{-1}$. 
  Thus $\varphi^2\equiv 1$ and since $A$ is odd we have $\varphi\equiv 1$ which 
  is equivalent to $\omega \in H^3(A,\TT)_\mathrm{ab}$.
\end{proof}
In other words, the restriction gives a split exact sequence 
\begin{align}
  \{0\} \longrightarrow H^3(\ZZ_2,\TT) %
  \stackrel{\displaystyle\mathop{\longrightarrow}^{\phantom{\mathrm{res}}}}{\displaystyle\mathop{\longleftarrow}_{\mathrm{res}}}
  H^3(\Dih(A),\TT) 
  \stackrel{\displaystyle\mathop{\longrightarrow}^{{\mathrm{res}}}}{\displaystyle\mathop{\longleftarrow}_{\phantom{\mathrm{res}}}}
   H^3(A,\TT)_\mathrm{ab} \longrightarrow
  \{0\}
\end{align}
where $H^3(A,\TT)_\mathrm{ab}\oplus H^3(\ZZ_2,\TT)\to H^3(\Dih(A),\TT)$ is given by Proposition \ref{prop:SomeDihAreMeptaplectic}.
Further, the proof of Lemma \ref{lem:H3ab} shows that 
$H^3(A,\TT)_\mathrm{ab}=H^3(A,\TT)^\tau\cong H^0(\ZZ_2,H^3(A,\TT))$ and 
$H^3(\ZZ_2,H^0(A,\TT))=H^3(\ZZ_2,\TT)$. 

For $\omega\in H^3(A,\TT)_\mathrm{ab}$ let us denote by $(G_{A,\omega},q_{A,\omega})$ a metric group, such that 
$Z(\Vect_A^\omega)$ is braided equivalent to $\cC(G_{A,\omega},q_{A,\omega})$. If $|A|$ is odd let 
$\bislotslot_{A,\omega}$ be the non-degenerate bicharactor on $A$ with $\bi gg_{A,\omega}^{-1}=q_{A,\omega}(g)$. 
Lemma \ref{lem:H3ab} implies that the assumption in Proposition 
\ref{prop:SomeDihAreMeptaplectic} that $Z(\Vect^\omega_A)$ is pointed is 
automatic. 
Thus we get the following result:
\begin{cor} 
  All twisted doubles of generalized dihedral groups $\Dih(A)$ with $A$ of odd 
  order are generalized metaplectic modular categories.
  
  Namely, $Z(\Vect_{\Dih(A)}^{\tilde\omega})$ is braided equivalent to 
  $\MP(G_{A,\omega},\bislotslot_{A,\omega},\nu)$, where 
  $\omega=\tilde\omega\restriction A$ and $\nu$ is the Frobenius--Schur 
  indicator
  $[\tilde\omega\restriction \ZZ_2]\in H^3(\ZZ_2,\TT)$ and 
  $(G_{A,\omega},\bislotslot_{A,\omega})$ as above.
\end{cor}
\begin{cor}
  A  Tambara--Yamagami category $\cF$ based on an abelian group  of odd order
  is group theoretical 
  if and only if $\cF\cong \TY(G_{A,\omega},\bislotslot_{A,\omega},\pm)$
  for some abelian group $A$ of odd order and some
  $\omega\in H^3(A,\TT)_\mathrm{ab}$.

  In this case, $\cF$ is Morita equivalent to 
  $\Vect^{\omega_\pm\times \bar\omega}_{\Dih(A)\times A}$ or equivalently
  $\Vect^{\omega_\pm}_{\Dih(A)}\boxtimes\Vect^{\bar\omega}_A$,
  where $[\omega_\pm]\in H^3(\Dih(A),\TT)$ is characterized by
  $[\omega_\pm\restriction A]=[\omega]$ and 
  $[\omega_\pm\restriction \ZZ_2]=[\pm]\in H^3(\ZZ_2,\TT)$, respectively.
\end{cor}
\begin{proof} 
  If $\TY(G,\bislotslot,\pm)$ is group theoretical then it admits an Lagrangian
  $\hat A$ and $\MP(G,\bislotslot,\pm)$ is braided equivalent to 
  $Z(\Vect^{\omega_\pm}_{\Dih(A)})$ as in 
    Proposition \ref{prop:SomeDihAreMeptaplectic}. 
  We only need to observe that we have braided equivalences: 
  \begin{align}
    Z(\TY(G_{A,\omega},\bislotslot,\pm)) 
      &\cong
    \MP(G_{A,\omega},\bislotslot,\pm)\boxtimes 
      \cC(G_{A,\omega},\bar q_{A,\omega})\\&\cong
    Z(\Vect^{\omega_\pm}_{\Dih(A)})\boxtimes Z(\Vect^{\overline\omega}_A)
  \end{align}
  which shows the if part and the second statement.
\end{proof}
\begin{cor}
  All twisted doubles of generalized dihedral groups of odd abelian groups 
  can be realized by a $\ZZ_2$-orbifold net $\A_L^{\ZZ_2}$ of a conformal net
  associated with a lattice.
\end{cor}

\subsection{Quantum double subfactors of Tambara--Yamagami categories and 
  Bisch--Haagerup subfactors}
\label{ssec:QuantumDoubleAndBischHaagerup}
\begin{prop}
  \label{prop:BischHaagerup}
  Let $A$ be an abelian group of odd order.
  Consider $G=\Dih(A)\times A$, and let $H=\langle \tau\rangle\cong \ZZ_2$ 
  and $\Delta(A)$ the diagonal embedding of $A$ in 
  $A\times A \subset \Dih(A)\times A$. 
  Then $G=\langle \Delta(A), H \rangle$.
  The Longo--Rehren subfactor $S\subset T$ associated with 
  $\TY(A,\bislotslot,\pm)$ is a
  Bisch--Haagerup subfactor $M^H\subset M\rtimes \Delta(A)$ for a $G$-action
  on $M=S\rtimes \ZZ_2$.
\end{prop}
\begin{proof}
  We use the conformal net realization of the Longo--Rehren subfactor.
  With $M=\A_L(I)\otimes\A_{\bar L}(I)$
  we have that the Longo--Rehren subfactor is of the form: 
  \begin{align}
    M^{\langle(\tau,1)\rangle} &\subset M \subset M\rtimes_{(g,g^{-1})} A 
  \end{align}
  and it follows that the action of $G \to \Out(M)$ is characterized by the 
  dual category of the inclusion $\A^{\ZZ_2}_L(I)\otimes\A_{\bar L}(I)
  \subset \A_L(I)\otimes\A_{\bar L}(I)\equiv M$,
  where $(\tau,1)=\tau\otimes \id $ and $(g,\bar g)=g\otimes \bar g$.

  We claim that with $H=\langle (\tau,1)\rangle\cong \ZZ_2$ and 
  $K = \langle (g,\bar g):g\in G\rangle\cong A$ we get 
  $G:=\langle K,H\rangle =\{(g,\bar h) : g\in \Dih(A), h\in A\}
    \cong \Dih(A)\times A$.
  Namely, since $(g,\bar g)(\tau,1) (g^{-1},\bar{g}^{-1}) (\tau,1) =(g^2,1)$ and
  $(g,\bar g)(\tau,1) (g,\bar{g}) (\tau,1) =(1,\bar{g}^2)$ and since $|A|$ is 
  odd we get  $\{(g,\bar h): g,h\in A\}\subset G$ and therefore it comes from a   $G$-kernel for some $[\omega]\in H^3(G,\TT)$.
  To show that it comes from a $G$-action, we have to show that the obstruction
  $[\omega]\in H^3(G,\TT)$ vanishes. 
  Because of the tensor product form $\omega$ is a product
  $H^3(\Dih(A),\TT)\oplus H^3(A,\TT)$
  and therefore vanishes by Proposition \ref{prop:TrivialDihCocycle}.
\end{proof}
\begin{cor} 
  Let $A$ be a cyclic group of odd order, then the Longo--Rehren subfactors
  of Tambara--Yamagami categories associated with $A$ are all conjugated.
\end{cor}
\begin{proof}
  Since $A$ is cyclic $H^2(A,\TT)$ and $H^2(\ZZ_2,\TT)$ vanish.
  Thus the Bisch--Haagerup subfactor associated with $A$ as 
  described in Proposition
  \ref{prop:BischHaagerup} is unique up to conjugacy. 
\end{proof}

It is straight forward to determine the (dual) principal graph for the 
Longo--Rehren subfactor associated with $\TY(A,\bislotslot,\pm)$. 
Let $n=2k+1=|A|$ and let
$A_+= \{h\in A : -h<h\}$ for some order on $A$, or equivalently, 
let $A=\{0\}\sqcup A_+\sqcup -A_+$. Thus we have $|A_+|=k$.
The dual principal graph $\Gamma'$ is given by
\begin{description}
  \item[even vertices ($M^K$ -- $M^K$ sectors)] There are $n(2+k)$ vertices which are given by
    $(\id,g),(\alpha,g)$ and $(\sigma_{(h,-h)},g)$, where $g\in A$ and $h\in A_+$.
  \item[odd vertices ($M\rtimes H$ -- $M^K$ sectors)] There are $n$ vertices which are given by
    $\iota(\id,g)=\iota(h,g+h)$, where $g,h\in A$.
  \item[edges]
      $\iota(\id,g)$  is connected to    $(\id,g),(\alpha,g)$
    and  $\{(\sigma_{(g-h,h-g)},h): h\in A, g-h\neq 0\}$ for all $g\in A$.
\end{description} 
The principal graph $\Gamma$ is given by
\begin{description}
  \item[even vertices ($M\rtimes H$ -- $M\rtimes H$ sectors)] There are $n^2+1$ 
    vertices given by
    $(g,h)$ and $(\rho,\rho)$, where $g,h\in A$.
  \item[odd vertices ($M\rtimes H$ -- $M^K$ sectors)] There are $n$ vertices given by
    $\iota(\id,g)=\iota(h,g+h)$, where $g,h\in A$.
  \item[edges]
    $\iota(\id,g)$ is connected to $\{(h,h+g):h\in A\}$ and $(\rho,\rho)$ for all $g\in A$.
\end{description} 

\begin{example} For $G=\ZZ_3$ the (dual) principal graph is given by
  \begin{align}
    \Gamma'&=
    \tikzmath[.5]{
      \draw (-1,0)--(1,0)--(0,1.732)--(-1,0);
      \draw (-1,0)--+(180:1) node (a) {}
        (-1,0)--+(240:1) node (b) {};
      \draw (1,0)--+(0:1)  node (d) {}
        (1,0)--+(-60:1)  node (e) {};
     \draw (0,1.732)--+(60:1)  node (g) {}
      (0,1.732)--+(120:1)  node (h) {};
      \fill 
        (a) circle (.1)
        (b) circle (.1)
        (d) circle (.1)
        (e) circle (.1)
        (g) circle (.1) 
        (h) circle (.1) node [above] {$\scriptstyle\ast$}
        (0,0) circle (.1)
        (0.5,.866) circle (.1)
        (-0.5,.866) circle (.1)
      ;
      \draw[fill=white] 
        (-1,0) circle (.1)
        (1,0) circle (.1)
        (0,1.732) circle (.1)
      ;
    }
    &
    \Gamma&=
    \tikzmath[.5]{
      \draw 
        (0,0) node (j) {} -- (0,1) node (o) {}
        (0,0)-- (-.866,-.5) node (p) {}
        (0,0)-- (.866,-.5) node (q) {}
      ;
      \draw (0,1)--+(90:1) node (a) {}
        (0,1)--+(30:1) node (b) {}
        (0,1)--+(150:1) node (c) {};
     \draw (.866,-.5)--+(330:1) node (d) {}
        (.866,-.5)--+(270:1) node (e) {}
        (.866,-.5)--+(30:1) node (f) {};
     \draw (-.866,-.5)--+(150:1) node (g) {}
        (-.866,-.5)--+(210:1) node (h) {}
        (-.866,-.5)--+(270:1) node (i) {};
      \fill 
        (a) circle (.1)
        (b) circle (.1) 
        (c) circle (.1) node [above] {$\scriptstyle\ast$}
        (d) circle (.1)
        (e) circle (.1)
        (f) circle (.1)
        (g) circle (.1)
        (h) circle (.1)
        (i) circle (.1)
        (j) circle (.1)
      ;
      \draw[fill=white] 
        (o) circle (.1)
        (p) circle (.1)
        (q) circle (.1)
      ;
    }
  \end{align}
  This graph also appeared in \cite{Bu2015}.
\end{example}

\subsection{Tambara--Yamagami realizations are $\ZZ_2$-twisted
and generalized orbifolds}
\label{ssec:TYareTwistedOrbifolds}
Let $\A$ be a conformal net with a proper $\ZZ_2$-action.
The orbifold net $\A^{\ZZ_2}$ has always the
\emph{trivial} $\ZZ_2$-simple current extension which recovers $\A$.
We call a non-trivial $\ZZ_2$-simple current extension of $\A^{\ZZ_2}$
a $\ZZ_2$-\myemph{twisted orbifold} of $\A$ \cf \cite[Section 3]{KaLo2006}
for twisted orbifolds of holomorphic nets.

We remember that if $\A$ is a conformal net and $K$ a finite hypergroup, 
then there is the notion of a proper action of $K$ on $\A$ and the fixed-point
net $\A^K$ is the \myemph{generalized orbifold} of $\A$, see \cite{Bi2016}.
\begin{prop}
  \label{prop:TwistedGeneralizedOrbifold}
  Let $\cF$ be a Tambara--Yamagami category associated with a group 
  $A$ of odd order. 
  Then there is a self-dual lattice $\Gamma$ and a proper $\Dih(A)$ action 
  on $\A_\Gamma$, such that there is a $\ZZ_2$-twisted orbifold of 
  $\A_\Gamma^{\Dih(A)}$, which realizes $Z(\cF)$.

  Namely, there is a lattice $L$, such that $\Gamma=L\oplus \hat A$,
  with an action of $\ZZ_2^2$, such that
  \begin{align}
    \Rep(\A_L^{\langle(1,0)\rangle})&\cong Z(\cF)\\
    \Rep(\A_L^{\langle(1,1)\rangle})&\cong Z(\Vect_{\Dih(A)})\\
    \Rep(\A_L^{\langle(0,1)\rangle})&\cong Z(\cF^\op)
  \end{align}
  with $\A_L^{\langle (1,1)\rangle} = \A_\Gamma^{\Dih(A)}$.

  Further,  $\A_L^{\langle(1,0)\rangle}$ and $\A_L^{\langle(0,1)\rangle}$
  are generalized orbifolds of $\A_\Gamma$, with respect to 
  proper actions of the hypergroup of the Tambara--Yamagami fusion rules 
  associated with $A$.
  In the same way $\A_L^{\ZZ_2^2}$ is a generalized orbifold with respect to the  hypergroup of the generalized Tambara--Yamagami fusion rules based on 
  $\Dih(A)$ in \eqref{eq:genTY}.
\end{prop}
\begin{proof}
  Assume $\cF\cong\TY(A,\bislotslot,\nu)$. 
  By Theorem \ref{thm:MPRealization}
  there are lattices $M$ and $\bar M$ and $\ZZ_2$-actions on $\A_M$ and 
  $\A_{\bar M}$, such that $\A^{\ZZ_2}_M$ and 
  $\A_{\bar M}^{\ZZ_2}$ realize
  $\MP(A,\bislotslot,\nu)$ and $\MP(A, \bislotslot^{-1},\nu)$, respectively.
  Choose $L=M\oplus \bar M$, then as in the proof of Proposition 
  \ref{prop:TwistedDoubleRealization} there is an action of $\Dih(A)$ on
  $\A_\Gamma$ with
  \begin{align}
    \A_{L}^{\langle(1,1)\rangle}&=\A_\Gamma^{\Dih(A)}\subset \A_\Gamma\,,
  \end{align}
  using the above 
  $\ZZ_2\times\ZZ_2$-action on $\A_{L}\cong \A_M\otimes\A_{\bar M}$.
  The rest follows immediately.
\end{proof}
\begin{rmk}
  There is a twisted version, where 
  \begin{align}
    \Rep(\A_L^{\langle(1,1)\rangle})&\cong Z(\Vect^{[\omega_-]}_{\Dih(A)})\\
    \Rep(\A_L^{\langle(0,1)\rangle})&\cong Z(\cF_-^\op)
  \end{align}
  with $\cF_-\cong \TY(A,\bislotslot,-\nu)$ for 
  $\cF\cong \TY(A,\bislotslot,\nu)$.
\end{rmk}
\begin{rmk}
  If we take $\Dih(A)$ for $|A|$ odd, we can consider 
  the element $c_\tau=\frac{1}{|A|}\sum_{g\in A} g\tau\in \CC[\Dih(A)]$.
  Then $c_\tau c_\tau=\frac1{|A|}\sum_{g\in A}g$ and 
  $K=\{A,c_\tau\}$ represents the $A$-Tambara--Yamagami hypergroup.
  Let us now assume $|A|$ is odd.
  The action $\alpha$ of $\Dih(A)$ from Proposition
  \ref{prop:TwistedGeneralizedOrbifold} gives an action 
  of $K$ by $\phi(c_\tau)=\frac{1}{|A|} \sum_{g\in A\tau} \alpha_{g}$, 
  but this action is not proper since 
  $\phi(c_\tau)$ fails to be extremal if $A$ is non-trivial.
  On the other hand, the two $\ZZ_2$-twisted orbifolds 
  of the $\Dih(A)$-orbifold $\A_{\Gamma}^{\Dih(A)}$ give proper actions of $K$
  and therefore ``generalized $K$-orbifolds''.

  It is enlightning to draw the lattice of intermediate nets
  for $\A_L^{\ZZ_2\times\ZZ_2}\equiv \A_\Gamma^{K_\mathrm{gen}}\subset \A_\Gamma$:
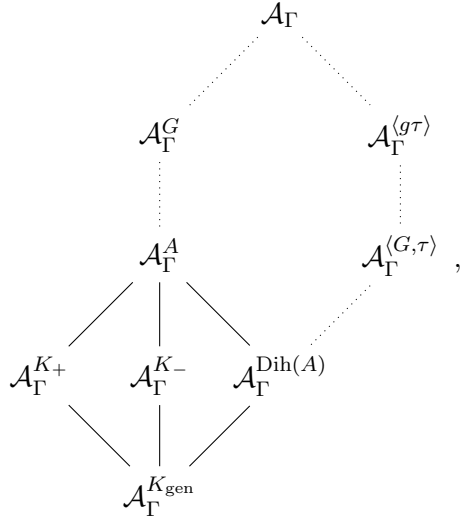
  $$
    \tikzmath[0.4]{
      \node (tr)  at (4,16) {$\A_{\Gamma}$};
      \node (tau)  at (8,12) {$\A_{\Gamma}^{\langle g\tau\rangle}$};
      \node (G)  at (0,12) {$\A_{\Gamma}^{G}$};
      \node (Gtau)  at (8,8) {$\A_{\Gamma}^{\langle G,\tau\rangle}$};
      \node (H)  at (0,8) {$\A_{\Gamma}^{A}$};
      \node (Kp)  at (-4,4) {$\A_{\Gamma}^{K_+}$};
      \node (Km)  at (0,4) {$\A_{\Gamma}^{K_-}$};
      \node (Dih)  at (4,4) {$\A_{\Gamma}^{\Dih(A)}$};
      \node (GTY)  at (0,0) {$\A_{\Gamma}^{K_\mathrm{gen}}$};
      \draw[dotted] (tr)--(tau)
(tr)--(G) (G)--(H) (tau)--(Gtau) (Gtau)--(Dih);
      \draw         (H)--(Kp) (H)--(Km) (H)--(Dih)
        (GTY)--(Kp) (GTY)--(Km) (GTY)--(Dih);
  }\,,
  $$
  where the dotted part has to be filled by the respective
  intermediate groups of $\Dih(A)$. 
  We note that for subgroups $G\subset A$ the net
  $\A_{\Gamma}^G$ is the conformal net 
  associated the lattice $L\oplus \hat A^G$,
  where $\hat A^G=\{\chi\in \hat A: \chi (g)=1 \text{ for all }g\in G\}$.
  The hypergroup $K_\mathrm{gen}$ is the hypergroup associated with the generalized Tambara--Yamagami category   based
  on $\Dih(A)$ (see Subsection \ref{ssec:GenTY})
 which can be seen as $K_+\times_{A} K_-$, namely a relative product over $A$ of two Tambara--Yamagami 
  hypergroups $K_\pm$ based on $A$.
  All solid lines are $\ZZ_2$-orbifolds.
\end{rmk}

\section{Generalized orbifolds and defects}
\label{sec:GeneralizedOrbifoldsAndDefects}
\subsection{Holomorphic nets from twisted orbifolds}
\label{ssec:HolomorphicNets}
The following is an analogue of the twisted orbifold construction in VOAs. 

\newcommand{\TO}{/\!/}
\begin{defi}
  Let $\A$ be a holomorphic net, and $G\leq \Aut(\A)$ a finite group.
  A holomorphic net $\cB$ is called a \myemph{twisted $G$-orbifold}
  of $\A$ if it is a holomorphic extension $\cB\supset  \A^G$.
\end{defi}
If $\A$ is holomorphic and $G\leq \Aut(\A)$ the category of 
$G$-twisted representations $\tRep G(\A)$ is tensor equivalent to 
$\Vect_G^{\omega}$ for some $[\omega]\in H^3(G,\TT)$.
More explicitly, by using $\alpha^+$-induction applied to $\A(I)^G\subset\A(I)$ for a fixed interval $I\subset S^1\setminus \{-1\}$
we get a $G$-kernel $\{[\alpha_g]:g\in G\} \subset \Out(\A(I))$ which can be 
lifted to $\Aut(G)$ if the associated obstruction in $H^3(G,\TT)$ vanishes.
Let us assume that the obstruction vanishes, \ie $\omega$ is a coboundary. 
In this case, we say that $G$ acts \myemph{anomaly free}.
Let us choose a trivialisation of $\omega$.
We can consider the crossed product $\A(I)\rtimes G$, which gives rise to a 
relatively local extension $\A\rtimes G$ of $\A^G$ which is a net on the 
universal cover of $S^1$ or on the restriction $S^1\setminus \{-1\}$. 
Namely, by definition the Q-system of $\A(I)^G\subset \A(I)\rtimes G$ is in 
$\Rep^I(A^G)$ which characterizes a non-local extension 
\cite{LoRe1995,LoRe2004}.
Let us denote by $\A^{\TO G}$ the intermediate net 
$\A^G\subset \A^{\TO G}\subset \A\rtimes G$ obtained by the left center 
construction \cite{BiKaLoRe2014}. 
The left and right center $\cB^\pm$ of a net $\cB$ on the real line are defined 
by
\begin{align}
  \cB^+(a,b) &:= \cB(a,b) \cap 
  \cB(-\infty,a)'\,,\\
  \cB^-(a,b) 
    &:= \cB(a,b) \cap \cB(b,\infty)'\,,
\end{align}
respectively, where we consider the net in the real line picture 
$\RR\cong S^1\setminus\{-1\}$. 
The nets are local and therefore extend to $S^1$. %
\begin{lem}
  Let $\A$ be holomorphic with an anomaly free action of $G$ as above. 
  The right center $[\A\rtimes G]^-$ of $\A\rtimes G\supset \A^G$ is $\A$
  and the left center $\A^{\TO G}$ is holomorphic. 
  The extensions  $\A^{\TO G}\supset \A^G$  and $\A\supset \A^G$
  are isomorphic if and only if $G$ is trivial.
\end{lem} 
\begin{proof} 
  Let $\psi_g\in (\iota_I,\iota_I\beta_{g,I}):= 
    \{\psi\in [\A\rtimes G](I):\psi x=\beta_g(x)\psi \text{ for all } 
    x\in\A(I)\}$.
  Then we have $\psi_g\in(\iota_K,\iota_J\beta_{g,K})$ for $K\supset I$.
  Let $I_{l,r}$ be the left and right component of $K\cap I'$, respectively.
  Since $\alpha_g$ is a right soliton, we have $\psi_gx=x\psi_g$ for 
  $x\in \A(I_l)$ and $\psi_gx=\alpha_g(x)\psi_g$ for $x\in \A(I_r)$.
  Thus $\A(I_l)
    =\A(I_l)\cap[\A\rtimes G](I+\RR_{\geq 0})'\subset [\A\rtimes G]^-(I_l)$.
  Thus $\A\subset [\A\rtimes G]^-$. 
  Since $\A$ is holomorpic we get equality.
  
  We have $\A^G(I_r)=\A(I_r)\cap [\A\rtimes G](I -\RR_{\geq 0})'
    =[\A\rtimes G]^+(I_r)\cap\A(I_r)$,
  thus the left center does coincide with $\A$ only if $G$ is trivial.
  
  The $\mu$ index of the left and right center coincide, but the right center 
  is $\A$, which is holomorphic, thus $\A^{\TO G}$ is holomorphic. 
\end{proof}
We note that $\A^{\TO G}$ itself can still be isomorphic to $\A$.
For example, if we consider the $\ZZ_2$-action $\A_{E_8}^{\ZZ_2}=\A_{D_8}$, 
then it follows that $\A_{E_8}^{\TO \ZZ_2}$ is isomorphic to $\A_{E_8}$.
It is well-believed that $\A_{E_8}$ is the only holomorphic conformal net with 
central charge $8$.
This is only a conjecture for conformal nets. 
If the conjecture was true it would imply that
$\A^{\TO G}$ is isomorphic to $\A_{E_8}$ for every finite group 
$G\leq \Aut(\A_{E_8})$ acting anomaly free.

\begin{lem} 
  $\A^G(I)\subset [\A\rtimes G](I)$ is irreducible.
\end{lem}
For an outer action of $G$ on a factor $M$, the inclusion 
$M^G\subset M\rtimes G$ is only irreducible for $G$ being trival. 
But in our case, we take fixed point with the gauge action and the 
crossed product with solitons.
\begin{proof}
  Let $\theta$ be the dual canonical endomorphism.
  We have to show that $\dim\Hom(\id,\theta)=1$, but we have a $G$-grading and 
  $\theta=\bigoplus_g\theta_g$, where $\theta_{e}$ is the dual canonical 
  endomorphism for $\A^G(I)\subset \A(I)$ and $\dim\Hom(\id,\theta_e)=1$.
\end{proof}

\begin{prop} 
  The extension $\A\rtimes G\supset \A$ defines an $\A^G$-topological defect
  (actually a phase boundary) between $\A$ and $\A^{\TO G}$ and this property 
  determines $\A^{\TO G}$ uniquely. 
  Namely, $\A^{\TO G}$  is characterized to be the unique holomorphic extension
  $\cB\supset\A^G$ giving $\A\rtimes G$ the structure of an $\A^G$-topological
  defect between $\A$ and $\cB$.

  There is action of the hypergroup $K_{\Rep(G)}$ on $\A^{\TO G}$, such 
  that $(\A^{\TO G})^{K_{\Rep(G)}}\cong \A^G$.
\end{prop}
We note that for $G$ abelian $K_{\Rep(G)}$ is a group and can be identified 
with the Pontryagin dual $\hat G$. 
On the other hand, if $G$ is non-abelian, we get an action of the genuine 
hypergroup $K_{\Rep(G)}$.

In general, for $\A$ a holomorphic net and $G\leq \Aut(\A)$ there is a unique
class $[\omega]\in H^3(G,\TT)$, such that $\tRep G(\A)\cong \Vect_G^\omega$.
For every $H\leq G$, such that $\omega\restriction H$ is a coboundary, \ie $H$ 
acts anomaly free, we can form $\A^{\TO H}$.
This choice is classified by $H^2(H,\TT)$.

On the other hand, a holomorphic twisted $G$-orbifolds of $\A$ is by definition
a holomorphic extension of $\A^G$.
These are in one-to-one correspondence with equivalence classes of 
Lagrangian Q-systems in $\Rep(\A^G)\cong Z(\Vect^\omega_G)$. 
But these are classified by the same data \cite{DaSi2017}, so the above 
construction gives all twisted orbifolds.

\begin{example}
  Let $(G,\bislotslot,\nu)$ be a triple consisting of an abelian group $G$ 
  of odd order, a non-degenerate symmetric bicharacter $\bislotslot$, 
  and a sign $\nu$.
  Let us consider the UMTCs $\cC_+=\MP(G,\bislotslot,\nu)$ and 
  $\cC_-=\rev{\MP(\bislotslot,\nu)}$.
  We find pairs of lattices $L_\pm$ and $\ZZ_2$-automorphisms $\alpha_\pm$ of 
  $\A_{L_\pm}$, such that $\Rep(\A_L^{\langle\alpha_\pm\rangle})$ is braided 
  equivalent to $\cC_\pm$, respectively. 
  Then there is a self-dual lattice $\Gamma=(L_+\lattimes L_-)\oplus G$ and an 
  action of $H=\hat G$ on $\A_\Gamma$, such that 
  $\A^H_\Gamma=\A_{L_+\lattimes L_-}$.
  This action extends to an action of $\Dih(H)$ on $\A_\Gamma$ by 
  Proposition \ref{prop:TwistedGeneralizedOrbifold}.
  Therefore we can form $\A^{\TO\Dih(H)}$ or more generally, $\A^{\TO K}$ for 
  $K\leq \Dih(H)$ which gives many examples twisted holomorphic orbifolds
  of conformal nets.
\end{example}
\begin{example}
  Let $L_1=A_2E_8$, $L_2=E_6E_8$ and $\Gamma=E_8^3$.
  There are $\ZZ_2$-actions such that we have braided equivalences
  \begin{align}
    \Rep(\A_{A_2E_8}^{\ZZ_2})&\cong \MP(\ZZ_3,\bislotslot,+)\,,&
    \Rep(\A_{E_6E_8}^{\ZZ_2})&\cong \MP(\ZZ_3,\overline{\bislotslot},+)\,,
  \end{align}
  where $\bi xy =e^{\frac{2\pi ixy}3}$. 
  We get an action of $S_3\cong \Dih(\ZZ_3)$ on $\A_{\Gamma}$ and can form 
  $\cB:=\A_{\Gamma}^{\TO S_3}$. 
  With the help of the computer program KAC \cite{ScKAC} we determine that the 
  weight one subspace should have dimension $456$ and that 
  $\cB$ should correspond to 64 in Schellekens' list 
  \cite{Sc1993}. 
  We therefore conjecture that $\A_{\Gamma}^{\TO S^3}$ is the  conformal net 
  $\A_{\Ni(D_{10}E_7^2)}$ associated with the lattice $\Ni(D_{10}E_7^2)$. 
  Here $\Ni(L)$ is the Niemeier lattice \cite{Ni1973} with root lattice $L$.
\end{example}

\subsection{Generalized orbifolds and hypergroup character tables}
\label{ssec:GeneralizedOrbifoldsAndCharacterTables}
Let $G$ be an abelian group of odd order which we see as a multiplicative group. Let $K=G\cup \{\tau\}$ the hypergroup associated with the 
Tambara--Yamagami fusion rules, \ie 
$\tau g=g\tau=\tau=\tau^{\ast}$ for $g\in G$ and 
$\tau^2=\frac{1}{|G|} \sum_{g\in G} g$.

Then we have the dual hypergroup 
$\hat K=\{1,\varepsilon, c_\chi:\chi\in \hat G\setminus\{1\}\}$ with 
$\varepsilon c_\chi=c_\chi\varepsilon =c_\chi$, $\varepsilon^2=1$ with 
relations:
\begin{align}
  c_\chi c_{\tilde\chi}&=
  \begin{cases} 
    \frac12(1+\varepsilon) &\chi=\tilde\chi^{-1}\,,\\
    c_{\chi\tilde\chi} & \text{otherwise}\,,
  \end{cases}
\end{align}
and the canonical pairing $\bislotslot_K\colon \hat K\times  K \to \CC$.
The character table is easily determined to be ($g\in G\setminus\{e\}$,
$\chi\in \hat G\setminus\{1\}$):
\begin{align}
  \label{eq:CharacterTable}
  \begin{array}{l|lll}
    \bislotslot_K &1 & g &\tau\\
    \hline
    1 & 1 & 1& 1\\
    \varepsilon & 1 &1 & -1\\
    c_\chi &1 & \chi(g)& 0
  \end{array}\,.
\end{align}
Namely, there is a self-dual lattice $\Gamma=L\bar L\oplus \hat G$ given by 
the glueing of two lattice $L$ and $\bar L$ and a proper hypergroup action of
$K$ on $\A_\Gamma$, such that $\Rep(\A_\Gamma^K)\cong Z(\cF)$ for some 
Tamabara--Yamagami category $\cF$ with $K_\cF=K$. 
The dual canonical endomorphism of the inclusion 
$\A_\Gamma^K(I)\subset \A_\Gamma(I)$ is given by
\begin{align}
  [\theta]&=\bigoplus_{\hat k\in \hat K}[\rho_{\hat k}]=
    [(0,0)]\oplus[(\alpha,0)]\oplus
    \bigoplus_{g\in G^\times}[(\theta_{(g,-g)},\bar g)]
\end{align}
with the correspondence 
\begin{align}
  [(0,0)] &\leftrightarrow 1&
  [(\alpha,0)] &\leftrightarrow \varepsilon&
  [(\sigma_{(g,-g)},\bar g)]&\leftrightarrow \chi(\slot)
    =\bi{g}{\slot} \quad (g\in G^\times)\,.
\end{align}
As in \cite{Bi2016} this  gives an action of $V\colon K \to B(\Hil_{\A_\Gamma})$
\begin{align}
  \Hil_{\A_L}&=\bigoplus_{k\in\hat K} \Hil_k\,,&
  V(k)&= \bigoplus_{\hat k \in \hat K}\bi {\hat k} k_K\,.
\end{align}
which extends to a $\ast$-representation of $\CC K$. 
With $M:=\A_\Gamma(I)\supset \iota(N):=\A_\Gamma^K(I)$ we have 
$M=\bigoplus_{\hat k\in\hat K}  M_{\hat k}$, where $M_{\hat k}=\iota(N)\psi_{\hat k}$ and 
$\psi_{\hat k}\in
\Hom(\iota,\iota\rho_{\hat k})$ is a charged intertwiner. We have
$M_\bullet\Omega\subset \Hil_\bullet$ which gives a ``hypergrading'' 
$m_km_\ell \in \bigoplus_{n\prec k\ell} M_n$ for $m_\bullet\in M_\bullet$.

\subsection{Defects and Kramer--Wannier duality}
\label{ssec:KramerWannier}
Let $\A$ be a completely rational conformal net 
which we see by restriction as a net on $\RR$. Let 
$\cB\supset\A $ be a possibly non-local extension
and $\Theta \in \Rep(\A)$ be the corresponding Q-system, see \cite{LoRe2004}.

By \cite{BiKaLoRe2014}
the irreducibles of the fusion category of 
$\cB$--$\cB$ sectors, or equivalently $\bim\Theta{\Rep(\A)}\Theta$ describes 
phase boundaries (or defects) on Minkowski space between $\cB_2$ and itself. 
Here $ \cB_2 \supset \A\otimes \bar\A$ is the full CFT on Minkowski space
coming from the full center construction of $\Theta$, see \cite{BiKaLo2014}.
By $\A\otimes \bar \A$ 
we denote the net on Minkowski space, which we identify with the product 
$\RR\times \RR$ of two light rays, defined by
$(\A\otimes\bar\A)(I\times J) = \A(I)\otimes \A(J)$. 
A \myemph{phase boundary} (between $\cB_2$ and $\cB_2$ which is transparent for $\A_2=\A\otimes\bar\A$) is a quadrilateral inclusion of nets 
$\A_2(O) \subset \cB_2^{l,r}(O) 
  \subset \langle \cB_2^{l}(O),\cB_2^{r}(O)\rangle=:\cD(O)$ on a 
common Hilbert space with $\cD(O)$ being a factor and 
$\A_2(O)\subset \cB_2^{l,r}(O)$ being isomorphic to $\A_2(O)\subset \cB_2(O)$,  
such that $\cB_2^l(O_l)$ commutes with $\cB_2^r(O_r)$ for $O_l$ space-like left of $O_r$, see \cite{BiKaLoRe2014} for details.

Invertible objects $\{\alpha_g\}_{g\in G}$ with 
$[\alpha_g][\alpha_h]=[\alpha_{gh}]$ for a finite group 
$G$ correspond to \myemph{group-like defects} or \myemph{gauge transformations}. 
An object $\rho$ with fusion rules 
\begin{align}
  [\rho][\bar\rho]&=\bigoplus_{g\in G} [\alpha_g] &
  \end{align}
correspond to a \myemph{duality defect}.
This can be seen as generalization of Kramers--Wannier duality of the Ising 
model
with fusion rules 
$[\sigma][\sigma]=[1]+[\varepsilon]$ and $[\varepsilon]^2=[1]$, see 
\cite{FrFuRuSc2004}.

Let $G$ be a finite abelian group of odd order. 
For $\cF=\TY(G,\bislotslot,\pm)$ Theorem
\ref{thm:AllOddTYs} gives
lattices $L,\bar L$ and a $\ZZ_2$-action on $\A_L$.
We can consider the net $\A_L^{\ZZ_2}\otimes \bar\A_{L}$ 
or alternatively, the ``reflected'' chiral net 
$\A_L^{\ZZ_2}\otimes \A_{\bar L}=\A_\Gamma^{K_\cF}$. Then there is a 
full CFT $\cB_2 \supset \A_2:=\A_L^{\ZZ_2}\otimes \bar\A_{L}$ corresponding 
to the holomorphic extension $\A_\Gamma\supset \A_\Gamma^{K_\cF}$.
\begin{prop}
  The phase boundaries in the sense of \cite{BiKaLoRe2014} which are defects 
  between $\cB_2$ and itself and are invisible for $\A_2$ 
  correspond to the elements of the Tambara--Yamagami hypergroup 
  $K=G\cup \{\tau\}$.
\end{prop}
\begin{proof} Namely, $\cB_2\supset \A_L^{\ZZ_2}\otimes\bar\A^{\ZZ_2}_L$ 
is the full center of the inclusion $\A_L\supset \A_L^{\ZZ_2}$
and the phase boundaries for $\cB_2\supset \A_L^{\ZZ_2}
\otimes \bar\A^{\ZZ_2}_L$ are given by $\A_L$--$\A_L$
sectors or equivalently irreducibles in $\bim{\ZZ_2}{\Rep(\A_L^{\ZZ_2})}{\ZZ_2}$.
It can be checked that the sectors in $\bim[+]{\ZZ_2}{\Rep(\A_L^{\ZZ_2})}{\ZZ_2}\cong\cF$
exactly correspond to phase boundaries that preserve the intermediate net $\A_2=\A_L^{\ZZ_2}\otimes \bar\A_L$.\end{proof}

The algebra $\cB_2(I)$ is generated by $\A_2(I)$ and charged intertwiners
$\{\psi_1,\psi_\varepsilon,\psi_\chi:\chi\in \hat G\setminus\{1\}\}$
with $\psi_i\in\Hom(\iota,\iota\rho_i)$, see \cite{BiKaLo2014}.
A phase boundary (condition) gives relations between the generators $\psi_\bullet^{l}$ and $\psi_\bullet^{r}$
of $\cB^{l,r}_2(O)$, respectively. 
Using the  character table \eqref{eq:CharacterTable} we get
the following relations:
\begin{itemize}
  \item A group like defect $g\in G$  correspond to a gauge automorphism
$\psi^{l}_\chi= \chi(g)\psi_\chi^{r}$ and $\psi^{l}_{1/\varepsilon}=
\psi^{r}_{1/\varepsilon}$.
\item The duality defect $\tau$, due to  
the zero entry in the character table \eqref{eq:CharacterTable},  gives independent fields 
$\sigma_\chi=\psi_{c_\chi}^{l}$ and $\mu_\chi=\psi_{c_\chi}^{r}$ 
while $\varepsilon:=\psi^{l}_\varepsilon= -\psi^{r}_\varepsilon$.
\end{itemize}
The ``Kramers--Wannier duality'' of the conformal Ising model (the unique full CFT $\cB\supset \Vir_{1/2}\otimes\overline{\Vir}_{1/2}$)
$$
  (1,\sigma,\varepsilon)\longleftrightarrow (1,\mu,-\varepsilon)
$$
generalizes to a duality from the defect $\tau$ given by 
$$
(1,\sigma_{\chi_2},\ldots,\sigma_{\chi_{|G|}},\varepsilon)
  \longleftrightarrow (1,\mu_{\chi_2},\ldots, \mu_{\chi_{|G|}},-\varepsilon)
$$
for the heterotic theory $\cB_2\supset\A_L^{\ZZ_2}\otimes \bar\A_L$ associated with $\cF$. 
We note that, because the Ising category is modular, the Ising hypergroup
is self-dual and there is a correspondence between fields and defects, namely 
both are indexed by $\ZZ_2\cup\{\tau\}\cong\{1,\sigma,\varepsilon\}$. 
For $G\cup\{\tau\}$ with $G$ of odd order this correspondence breaks down. 
The defects are indexed by $G\cup \{\tau\}$ while the fields are indexed by
$\{1,\sigma_{\chi_2},\ldots,\sigma_{\chi_{|G|}},\varepsilon\}$.

It seems interesting to study the above dualities in statistical physics models on a lattice in terms of high/low temperature duality. 
We believe that such a duality may arise from gauging certain spin models. 
The golden way would be to find corresponding lattice models
whose continuum limit at criticality recovers the in this section constructed 
conformal field theories.
Such models might be easy enough to be interesting for physical implementations and applications in topological quantum computing.

\def\cprime{$'$}\newcommand{\noopsort}[1]{}
\address
\end{document}